\documentclass[10pt,reqno]{amsart}

\usepackage[a4paper,left=35mm,right=35mm,top=30mm,bottom=30mm,marginpar=25mm]{geometry}

\usepackage{amsmath}
\usepackage{amssymb}
\usepackage{amsthm}
\usepackage[utf8]{inputenc}
\usepackage{eurosym}
\usepackage{graphicx}
\usepackage{hyperref}
\usepackage{dsfont}
\usepackage{dsfont}
\usepackage{xcolor,colortbl}
\usepackage{comment}

\allowdisplaybreaks

\usepackage{hyperref}

\usepackage{ifthen}

\makeindex

\newcommand{\supp}{\operatorname{supp}}

\newcommand{\Rr}{{\mathbb{R}}}

\newcommand{\Nn}{{\mathbb{N}}}

\newcommand{\Ll}{{\mathcal{L}}}

\newcommand{\Aa}{{\mathcal{A}}}

\newcommand{\Ii}{{\mathcal{I}}}

\newcommand{\Pp}{{\mathcal{P}}}
\newcommand{\Uu}{{\mathcal{U}}}

\newcommand{\Ww}{{\mathcal{W}}}
\newcommand{\Mm}{{\mathcal{M}}}

\def\dx{{\rm d}x}
\def\dy{{\rm d}y}

\def\dt{{\rm d}t}
\def\ds{{\rm d}s}
\def\leq{\leqslant}
\def\geq{\geqslant}

\numberwithin{equation}{section}

\newtheoremstyle{thmlemcorr}{10pt}{10pt}{\itshape}{}{\bfseries}{.}{10pt}{{\thmname{#1}\thmnumber{
			#2}\thmnote{ (#3)}}}
\newtheoremstyle{thmlemcorr*}{10pt}{10pt}{\itshape}{}{\bfseries}{.}\newline{{\thmname{#1}\thmnumber{
			#2}\thmnote{ (#3)}}}
\newtheoremstyle{defi}{10pt}{10pt}{\itshape}{}{\bfseries}{.}{10pt}{{\thmname{#1}\thmnumber{
			#2}\thmnote{ (#3)}}}
\newtheoremstyle{remexample}{10pt}{10pt}{}{}{\bfseries}{.}{10pt}{{\thmname{#1}\thmnumber{
			#2}\thmnote{ (#3)}}}
\newtheoremstyle{ass}{10pt}{10pt}{}{}{\bfseries}{.}{10pt}{{\thmname{#1}\thmnumber{
			A#2}\thmnote{ (#3)}}}

\theoremstyle{thmlemcorr}
\newtheorem{theorem}{Theorem}
\numberwithin{theorem}{section}

\newtheorem{corollary}[theorem]{Corollary}
\newtheorem{proposition}[theorem]{Proposition}

\theoremstyle{thmlemcorr*}
\newtheorem{theorem*}{Theorem}
\newtheorem{lemma*}[theorem]{Lemma}
\newtheorem{corollary*}[theorem]{Corollary}
\newtheorem{proposition*}[theorem]{Proposition}
\newtheorem{problem*}[theorem]{Problem}
\newtheorem{conjecture*}[theorem]{Conjecture}

\theoremstyle{defi}

\newtheorem{hyp}{Assumption}

\newtheorem{problem}{Problem}

\theoremstyle{remexample}
\newtheorem{remark}[theorem]{Remark}

\newtheorem{pro}[theorem]{Proposition}
\newtheorem{cor}[theorem]{Corollary}

\theoremstyle{ass}

\usepackage{bm}
\usepackage{caption}
\usepackage{subcaption}
\usepackage[linesnumbered,ruled]{algorithm2e}

\begin{document}

\title{A Variational  Approach For Price Formation Models
	In One Dimension}

\author{Yuri Ashrafyan}\thanks{King Abdullah University of Science and Technology (KAUST), CEMSE Division, Thuwal 23955-6900. Saudi Arabia. e-mail: yuri.ashrafyan@kaust.edu.sa}%
\author{Tigran Bakaryan}\thanks{King Abdullah University of Science and Technology (KAUST), CEMSE Division, Thuwal 23955-6900. Saudi Arabia. e-mail: tigran.bakaryan@kaust.edu.sa}%
\author{Diogo Gomes}\thanks{King Abdullah University of Science and Technology (KAUST), CEMSE Division, Thuwal 23955-6900. Saudi Arabia. e-mail: diogo.gomes@kaust.edu.sa}%
\author{Julian Gutierrez}\thanks{King Abdullah University of Science and Technology (KAUST), CEMSE Division, Thuwal 23955-6900. Saudi Arabia. e-mail: julian.gutierrezpineda@kaust.edu.sa}%




\thanks{
      The authors were partially supported by King Abdullah University of Science and Technology (KAUST) baseline funds and KAUST OSR-CRG2021-4674.
}
\date{\today}
\begin{abstract}

In this paper, we study a class of first-order mean-field games (MFGs) that model price formation. Using Poincar\'e Lemma, we eliminate one of the equations and obtain a variational problem for a single function. This variational problem offers an alternative approach for the numerical solution of the original MFGs system. We show a correspondence between solutions of the MFGs system and the variational problem. Moreover, we address the existence of solutions for the variational problem using the direct method in the calculus of variations. We end the paper with numerical results for a linear-quadratic model.

\end{abstract}

\maketitle


\section{Introduction}
Here, we consider the numerical solution of the first-order mean-field games (MFGs) system introduced in \cite{gomes2018mean} to model price formation. The solution to this system determines the price $\varpi$ of a commodity with supply $Q$ when a large group of rational agents trades that commodity. The original price problem reads as follows:

\begin{problem} \label{PMFG} 
Suppose that $m_0\in \mathcal{P}(\Rr)$, $H\in C^1(\Rr)$, and $Q$, $V$, and $u_T$ are continuous. Assume further that $H$ is uniformly convex. Find $u,m: [0,T]\times \Rr \to \Rr$ and $\varpi:[0,T]\to \Rr$ satisfying  $m\geq 0$,
\begin{equation}\label{eq:MFG system}
\begin{cases}
-u_t +H(\varpi + u_x )+V(x)=0 & [0,T]\times \Rr,
\\
m_t - \left(H^\prime(\varpi + u_x)m\right)_x =0 & [0,T]\times \Rr,
\\
-\int_{\Rr} H^\prime(\varpi + u_x)m \dx = Q(t) & [0,T], 
\end{cases}
\end{equation}
and
\begin{equation}\label{boundaryP}
	\begin{cases} 
	m(0,x)=m_0(x) &  \\ u(T,x)=u_T(x)& 
	\end{cases} x \in \Rr.
\end{equation} 
\end{problem}  

The existence of solutions $(u,m,\varpi)$ to the previous problem was proved in \cite{gomes2018mean}. The first equation is solved in the viscosity sense by the value function of a typical player $u\in C([0,T]\times \Rr)$. The second equation is solved in the distributional sense by the probability distribution of the agents, $m\in C([0,T],\Pp(\Rr))$. The price $\varpi$ is a continuous function on $[0,T]$. 

Price formation models offer a load-adaptive pricing strategy relevant in energy markets. For instance, \cite{ATM19} and \cite{alasseur2021mfg} modeled intraday electricity markets, obtaining a price from the solution of forward-backward equations. In  \cite{FTT20}  and  \cite{fujii2021equilibrium} authors  studied the effects of a major player in the market.  The latest paper considered $N$-agent setting.  A deterministic  $N$-agent price model was studied in \cite{SummerCamp2019}. A MFG  model of homogeneous agents for the electricity markets was considered in  \cite{feron2021price}. 
 In \cite{aid2020equilibrium}, the price equilibrium is obtained for a finite number of agents who optimally control their production and trading rates in order to satisfy a demand subjected to common noise. Stackelberg games for price formation under revenue optimization were proposed in \cite{BS02} and \cite{BS10}, and Cournot models in \cite{TBD20}. A MFG of optimal switching was presented in \cite{AidDumitrescuTankov2021} to model the transition to renewable energies. Other works incorporating market-clearing conditions are \cite{JSF20} and \cite{FT20}, the former specialized to Solar Renewable Energy Certificate Markets and the latter in exchange markets. The stochastic supply case was studied in \cite{GoGuRi2021}, where authors obtained a price from a Lagrange multiplier rule for the balance constraint.

The standard MFG system exhibits a coupling of two partial differential equations with initial and terminal conditions (see for example \cite{cardaliaguet2018short}). Several numerical methods have been proposed to solve these MFG systems. Finite differences schemes and Newton-based methods were introduced in \cite{CDY} and \cite{DY}. A recent survey can be found in \cite{Achdou2020}. Optimization methods and Fourier series approximations were proposed in \cite{Yang2021}. Machine learning methods have been studied in \cite{carmona2021convergenceergo},  \cite{carmona2021convergence}, \cite{ruthotto2020machine}, and \cite{Line2024713118}. However, the MFG system \eqref{eq:MFG system}-\eqref{boundaryP} not only couples a forward equation for $m$ with a backward equation for $u$ but also determines the coupling term $\varpi$ through an integral constraint, which is the third equation in \eqref{eq:MFG system}. Therefore, the numerical approximation of the solution $(u,m,\varpi)$ of Problem \ref{PMFG} is challenging, and the main application of our methods is a novel numerical scheme for Problem \ref{PMFG}.

The word Potential in MFGs is used in two unrelated contexts. Potential MFGs (\cite{ll2}, \cite{cardaliaguet2018short}, \cite{OrPoSa2018}) are MFG systems given by the first-order optimality conditions of a minimization problem. Previously, standard optimization techniques were used for its numerical solution (\cite{bonnans2021discrete}). In contrast, our potential approach relies on the structure of the continuity equation and Poincar{\'e} lemma (\cite{Csato2011ThePE}, Theorem 1.22). We introduce a potential functional that integrates the transport equation in \eqref{eq:MFG system}. 

Poincar{\'e} lemma was used for the continuity equation in \cite{DRT2021Potential} for the MFG planning problem. The authors obtained a variational problem for a potential function by eliminating one of the equations in the MFG system. Moreover, the solution $(u,m)$ of the planning MFG can be recovered using only the solution of the variational problem. The structure of the MFG planning problem differs from that in Problem \ref{PMFG} in two critical aspects: the initial-terminal conditions and the way the constraint couples the equations.

In Section \ref{sec:Derivation of the variaitonal problem}, we use the existence result for Problem \ref{PMFG} provided in \cite{gomes2018mean} to formally obtain a potential function, $\varphi$. We show that \eqref{eq:MFG system} corresponds to the Euler-Lagrange equation of a constrained variational problem depending on $\varphi$. To introduce this problem, let $F$ be the  Legendre transform of $H$; that is,
\begin{equation}\label{eq:Legendre transform}
	F(y)=\sup_{p \in \Rr}  \left[  p y - H(p) \right], \quad y \in \Rr,
\end{equation}
and let $L:\Rr\times\Rr_0^+\to \Rr_0^+$ be given by
\begin{equation}\label{L def}
L(z,y) = \begin{cases} F\left( \frac{z}{y}\right) y, & (z,y) \in \Rr\times\Rr^+,
\\
+\infty, & z\neq 0,\, y=0,
\\
0, & z=0, \,y= 0. \end{cases}
\end{equation}
The constrained variational problem is
\begin{problem} \label{problem:Variational} 
Suppose that $m_0\in \mathcal{P}(\Rr)$, $H$ is uniformly convex, and $Q$, $V$, and $u_T$ are continuous. Find $\varphi: [0,T]\times \Rr \to \Rr$ that minimizes the functional
\begin{equation*}
\varphi \mapsto \int_0^T \int_{\Rr} L(\varphi_t,\varphi_x) -V(x)\varphi_x -u^\prime_T(x) \varphi_t~\dx\dt, 
\end{equation*}
over the set of functions such that $\varphi_x(t,\cdot)$ is a probability density on $\Rr$ for $t\in[0,T]$, $\varphi_x(0,\cdot)=m_0(\cdot)$, and satisfying
\begin{equation}\label{eq-bal-in-phi}
		\int_{\Rr} \varphi(t,x) - M_0(x) \dx = -\int_0^t Q(s)\ds, \quad  t\in [0,T].
\end{equation}
\end{problem} 
We work under assumptions similar to those in \cite{gomes2018mean} used to prove the existence and uniqueness of solutions to Problem \ref{PMFG}. The precise statement of our assumptions is presented in Section \ref{sec:Assumption}. We rigorously study Problem \ref{problem:Variational} in Section \ref{sec:Variational Approach compact}, where we show that its formulation is independent of the solution $(u,m,\varpi)$ of Problem \ref{PMFG}, and relies only on problem data. In Section \ref{sec:Price as lagrangem}, we obtain the existence of a price $\varpi$ in \eqref{eq:MFG system} as a Lagrange multiplier, and we establish the following connection between solutions of Problems \ref{PMFG}  and \ref{problem:Variational}:
	\begin{theorem}\label{pro-connection}
Suppose that $\varphi\in C^2([0,T]\times\Rr)$ solves Problem \ref{problem:Variational} . Then, the solution $(u,m,\varpi)$ of Problem \ref{PMFG} admits the representation
\begin{align*}
\begin{cases}
	 u(t,x) = u_T(x) - \int_t^T H\left(F'\left(\frac{\varphi_t(s,x)}{\varphi_x(s,x)}\right)\right) \ds - (T-t)V(x),\quad &(t,x)\in [0,T]\times \Rr
\\
m(t,x)=\varphi_{x}(t,x), \quad &(t,x)\in [0,T]\times \Rr
\\
 \varpi(t) = w_T-\int_t^T w(s)\ds, \quad &t\in[0,T], 	
\end{cases}
\end{align*}
where $w_T = \int_{\Rr} (L^*_z(T,y)-u'_T(y))\varphi_x(T,y) \dy$, 
\[
w(s) = \int_{\Rr} \bigg( \left(L^*_z(s,y)\right)_t +  \left(L^*_y(s,y)-V(y)\right)_x \bigg) \varphi_x(s,y) \dy , \quad s\in[0,T],
\]
and $L^*_z$ and $L^*_y$ are defined in \eqref{eq:Lzdef} and \eqref{eq:Lydef}, respectively.
\end{theorem}
Because Problem \ref{problem:Variational} is a convex minimization problem, we approximate its solution $\varphi$ by using standard optimization methods. Furthermore, using the approximations for $\varphi$ and Theorem \ref{pro-connection}, we obtain  efficient approximation methods for the solution to Problem \ref{PMFG}. In Section \ref{sec:numerical results}, we illustrate the implementation of our approach for the linear-quadratic setting, for which explicit formulas are provided in \cite{gomes2018mean} that can be used as benchmarks. For all these benchmarks, our numerical method provides accurate approximations.

\section{Derivation of the variational  problem}\label{sec:Derivation of the variaitonal problem}
In this section, we present a formal derivation of the variational problem for the potential function using the solution of the MFGs system. The precise assumptions we work with are stated in Section \ref{sec:Assumption}. The rigorous statement of the variational problem is given in Section \ref{sec:Variational Approach compact}, where we no longer rely on the solution of the MFGs system.

Let $(u,m,\varpi)$ solve Problem \ref{PMFG} with $m>0$. 
Then, the second equation in \eqref{eq:MFG system} can be
written as 
\[
	\text{div}_{(t,x)}\left( m , - H^\prime(\varpi + u_x)m \right) = 0, \quad [0,T]\times \Rr.
\]
The previous equation combined with Poincar{\'e} lemma (see \cite{Csato2011ThePE}, Theorem 1.22) gives the existence of a function (the potential) $\varphi : [0,T]\times \Rr \to \Rr$ such that 
\begin{equation}\label{eq: potential relations}
\begin{cases}
 m=	\varphi_x, \\ H^\prime(\varpi + u_x)m=\varphi_t.
\end{cases}
\end{equation}
Because $H$ is uniformly convex, $H^{\prime}$ is strictly monotone. Therefore, by \eqref{eq:Legendre transform}, we have
\begin{equation}\label{eq: Legendre tr}
F^\prime(y)=\left( H^{\prime}\right)^{-1}(y).
\end{equation}
Hence, from  the second  equation in \eqref{eq: potential relations}, we deduce that
\begin{equation*}
	u_x  = F^{\prime} \left( \frac{\varphi_t }{\varphi_x}\right)-\varpi.
\end{equation*}

If $V\in C^1(\Rr)$, and $u$ is twice differentiable, we differentiate the Hamilton-Jacobi equation in  \eqref{eq:MFG system} with respect to $x$ to obtain
\[
	-(u_x)_t + \left(H(\varpi + u_x ) \right)_x +V^\prime=0.
\]
Thus, the system \eqref{eq:MFG system} in terms of $\varphi$ is reduced to the following two equations
\begin{equation}\label{eq: Euler-Lagrange wrt potential}
	\begin{cases}
	-\left(F^{\prime} \left( \frac{\varphi_t }{\varphi_x}\right)-\varpi \right)_t + \left(H\left(F^{\prime} \left( \frac{\varphi_t }{\varphi_x}\right)\right)\right)_x+V^\prime=0 & [0,T] \times \Rr,
	\\
	-\int_{\Rr}  \varphi_t + Q \varphi_x ~ \dx =0 & t \in  [0,T],
	\end{cases}
\end{equation}
with initial condition, $\varphi(0,x)=\int_{-\infty}^x m_0(y)\dy\,\ x \in \Rr$, and terminal condition
\begin{equation}\label{boundary-in-phi}
 F^{\prime} \left( \frac{\varphi_t(T,x) }{\varphi_x(T,x)}\right)-\varpi(T)=u^\prime_T(x) \quad x \in \Rr.
\end{equation}

\begin{remark}\label{rem-phi-in-u-m}
Notice that the initial condition
implies that $\varphi_x(0,x)=m_0(x)$, $x\in\Rr$, which is the first equation in \eqref{boundaryP}. Moreover, we have the following explicit formula for $\varphi$ in terms of the solution $(u,m,\varpi)$ of \eqref{eq:MFG system} and \eqref{boundaryP} 
\begin{equation}\label{eq:Potential in terms of MFGs}
	\varphi(t,x)=\int_{-\infty}^{x} m_0(y)\dy + \int_0^t H'(\varpi(s)+u_x(s,x))m(s,x)\ds,\quad (t,x) \in [0,T]\times \Rr.
\end{equation}
Therefore, the potential function $\varphi$, which in principle has a closed formula arising from the solution of \eqref{eq:MFG system} and \eqref{boundaryP}, can be characterized using the initial condition with $m_0$, \eqref{eq: Euler-Lagrange wrt potential} and \eqref{boundary-in-phi}, which depend only, up to $\varpi$, on problem data. 
\end{remark}

\begin{remark}
Notice that the first equation in \eqref{eq: Euler-Lagrange wrt potential} shows that the expression 
\[
	-\left(F^{\prime} \left( \frac{\varphi_t }{\varphi_x}\right) \right)_t + \left(H\left(F^{\prime} \left( \frac{\varphi_t }{\varphi_x}\right)\right)\right)_x+V^\prime
\]
is independent of $x\in\Rr$, so it is a function of time only and equal to $\dot{\varpi}$. Similarly, \eqref{boundary-in-phi} shows that
\[
	F^{\prime} \left( \frac{\varphi_t(T,x) }{\varphi_x(T,x)}\right)-u^\prime_T(x) 
\]
is independent of $x\in\Rr$, and equal to the constant $\varpi(T)$. Because any numerical method to compute $\varphi$ provides an approximation of the value $\varphi(t,x)$, we can not expect the numerical approximation to be independent of $x$ in \eqref{eq: Euler-Lagrange wrt potential} and \eqref{boundary-in-phi}. Therefore, we can not rely on these formulas to recover $\varpi$ using an approximation of $\varphi$. In Section \ref{sec:Price as lagrangem}, we provide a formula approximating $\varpi$ that averages the dependence on $x$, and thus, can be implemented with any approximation of the potential.
\end{remark}

Next, consider the functional
\begin{equation}\label{functional-v}
	 	\begin{split}
	 	\int_0^T \int_{\Rr} L(\varphi_t,\varphi_x)-V\varphi_x \dx\dt  - \int_{\Rr} u^\prime_T(x)(\varphi(T,x)-\varphi(0,x)) ~\dx
	 	\end{split}
\end{equation}
subject to $\int_{\Rr} -(\varphi_t + Q \varphi_x)\dx=0$ on $[0,T]$, and with initial condition $\varphi(0,x)=\int_{-\infty}^x m_0(y)\dy$. Using the augmented functional associated with the constraint $\int_{\Rr} -(\varphi_t + Q \varphi_x)\dx=0$, we show that \eqref{eq: Euler-Lagrange wrt potential} is an Euler-Lagrange equation. Thus, we introduce a Lagrange multiplier $\varpi:[0,T]\to \Rr$ for the integral constraint, and we define
\begin{equation}\label{eq: functional}
\begin{split}
\tilde{I}[\varphi,\varpi]:=\int_0^T \int_{\Rr} L(\varphi_t,\varphi_x)   -\varpi\left(  \varphi_t+Q\varphi_x~ \right) -V\varphi_x-u^\prime_T(x) \varphi_t~\dx\dt   
\end{split}
\end{equation}
with initial condition $\varphi(0,x)=\int_{-\infty}^x m_0(y)\dy$. By considering critical points $(\varphi,\varpi)$ of the previous functional, we obtain that \eqref{eq: Euler-Lagrange wrt potential} is the corresponding Euler-Lagrange equation, with the natural boundary condition \eqref{boundary-in-phi}.

\begin{pro}\label{Pro: EL} 
Let $(\varphi,\varpi)$ be a critical point of the functional \eqref{eq: functional} over $C^2([0,T]\times\Rr)\times C^1([0,T])$ satisfying  $\varphi(0,x)=\int_{-\infty}^x m_0(y)\dy$. Assume further that $\varphi_x>0$. Then, the corresponding Euler-Lagrange equation is equivalent to \eqref{eq: Euler-Lagrange wrt potential}-\eqref{boundary-in-phi}.
\end{pro}
\begin{proof}

Let $(\varphi,\varpi)$ be a critical point of \eqref{eq: functional}. Taking $(\beta^1,\beta^2)\in C^1_c((0,T]\times\Rr)\times C([0,T])$, we have
\begin{equation}\label{eq-diff}
	\left. \frac{d}{d\varepsilon}\tilde{I}[(\varphi,\varpi) + \varepsilon (\beta^1,\beta^2)] \right|_{\varepsilon=0}=0.
\end{equation}
The previous identity implies that 
\begin{equation}\label{eq-LE}
	-\left( L_z(\varphi_t,\varphi_x)  - \varpi \right)_t - \left(L_y(\varphi_t,\varphi_x) \right)_x +V^\prime =0,
\end{equation}
and 
\begin{equation}\label{boundary-L}
	 L_z(\varphi_t(T,x),\varphi_x(T,x)) - \varpi(T) - u'_T(x)=0
\end{equation}
on $[0,T]\times \Rr$. Because $\varphi_x>0$, \eqref{L def} gives
\begin{equation}\label{L-dif-ident}
\begin{split}
& L_z(\varphi_t,\varphi_x)=F^\prime\left( \frac{\varphi_t}{\varphi_x}\right),
 \\
 & \left(  L_y(\varphi_t,\varphi_x)\right) _x=\left( -\frac{\varphi_t}{\varphi_x}F^\prime\left( \frac{\varphi_t}{\varphi_x}\right)+F\left( \frac{\varphi_t}{\varphi_x}\right)\right)_x= -\frac{\varphi_t}{\varphi_x}\left( F^\prime\left( \frac{\varphi_t}{\varphi_x}\right)\right)_x .
\end{split}
\end{equation} 
Notice that, by \eqref{eq: Legendre tr}, we have 
\begin{equation}\label{H-identity}
\left(H\left(F^{\prime} \left( \frac{\varphi_t }{\varphi_x}\right)\right)\right)_x=\frac{\varphi_t}{\varphi_x}\left( F^\prime\left( \frac{\varphi_t}{\varphi_x}\right)\right)_x. 
\end{equation}
Combining the identities in \eqref{L-dif-ident} with \eqref{H-identity} and using \eqref{eq-LE}, we deduce the first equation in \eqref{eq: Euler-Lagrange wrt potential}. Using the first identity of \eqref{L-dif-ident} in \eqref{boundary-L}, we obtain  \eqref{boundary-in-phi}.

Finally, taking $\beta^1\equiv 0$ in \eqref{eq-diff}, we obtain
\[
	\int_0^T\left( \int_{\Rr} -(\varphi_t + Q \varphi_x)\dx \right)  \beta^2~ \dt =0,
\]
where $\beta^2$ is arbitrary. Thus, the continuity of the map $t\mapsto \int_{\Rr} -(\varphi_t + Q \varphi_x)\dx$ implies the second equation in \eqref{eq: Euler-Lagrange wrt potential}. \qedhere
\end{proof}

In Section \ref{sec:Price as lagrangem}, we address the existence of the price $\varpi$ as a Lagrange multiplier associated with a minimizer of \eqref{functional-v}. 

\section{Assumptions}\label{sec:Assumption}
In this section, we state the assumptions to prove the existence of minimizers of the functional \eqref{functional-v}. This set of assumptions is similar to the ones introduced in \cite{gomes2018mean} to guarantee the existence and uniqueness of $(u,m,\varpi)$ solving \eqref{eq:MFG system} and \eqref{boundaryP}.

%

The following two assumptions require standard growth and convexity properties for $H$. 

\begin{hyp}\label{hyp: grow-L} 
There exist constants, $c>0$ and $p>1$, such that  the Legendre-Fenchel transform of $H$, the function $F$ in \eqref{eq:Legendre transform},  satisfies
\begin{equation*}
	F(v)\geq c|v|^p.
\end{equation*}
\end{hyp}

\begin{hyp}\label{hyp: H convex}
For all $x\in\Rr$, the map $p\mapsto H(p)$ is 
uniformly convex; that is, there exists a constant $\kappa >0$ such that $H''(p)\geq \kappa$ for all $p\in\Rr$. Moreover, there exists a positive constant, $C$, such that $|H^{\prime\prime\prime}|\leq C$.
\end{hyp}

For the supply, to simplify, we assume it is a smooth function of time.

\begin{hyp}
	The supply function, $Q$, is $C^\infty([0,T])$. 
\end{hyp}

The following assumption is technical and was used in \cite{gomes2018mean} to get bounds for the price. 

\begin{hyp}\label{hyp: V-uT Lipschitz 2nd D bounded DS}
The potential $V$, the terminal cost $u_T$, the initial density function $m_0$ are $C^2(\Rr)$ functions and  $V$, $u_T$ are  globally Lipschitz. Furthermore,  there exists a constant $C>0$ such that
\[
	|V''|\leq C, \quad |u_T''|\leq C, \quad |m_0''|\leq C.
\]
\end{hyp}

The following condition guarantees the uniqueness of solutions of \eqref{eq:MFG system} and \eqref{boundaryP}.

\begin{hyp}\label{hyp: V-uT convex DS}
	The potential $V$ and the terminal cost $u_T$ are convex.
\end{hyp}

Finally, because we are interested in problems where agent's assets are bounded, we require the following assumption on $m_0$. This assumption further simplifies some technical points in the presentation.

\begin{hyp}\label{hyp: m_0-comp}
The initial density function $m_0$ has compact support; that is, there exists $R_0>0$ such that $\supp(m_0) \subset [-R_0,R_0]$. 
\end{hyp}

\section{The variational approach}\label{sec:Variational Approach compact}
Here, we examine a variational problem associated with the MFG system \eqref{eq:MFG system}-\eqref{boundaryP} continuing the formal derivation in Section \ref{sec:Derivation of the variaitonal problem}. This problem is obtained by minimizing the functional \eqref{functional-v} in a suitable class of admissible functions. We study the existence and uniqueness of solutions to this variational problem. In Section \ref{sec:Price as lagrangem}, we establish a formula representing the solution to the MFG system \eqref{eq:MFG system}-\eqref{boundaryP}, in terms of the solution to this variational problem. 

First, we recall that, under Assumptions \ref{hyp: grow-L}-\ref{hyp: V-uT convex DS}, Theorem 1 in \cite{gomes2018mean} gives existence and uniqueness of solutions $(u,m,\varpi)$ to Problem \ref{PMFG}, where $m\in \mathcal{P}(\Rr) \cap C([0,T]\times\Rr)$. Moreover, $u$ is  a viscosity solution to the first equation in \eqref{eq:MFG system}, Lipschitz continuous and semi-concave in $x$, and  $u_x$, $u_{xx}$, $m$ are bounded. Furthermore, by the results in \cite{gomes2021duality}, $\varpi$ is Lipschitz continuous.


\subsection{Preliminary results for the continuity equation}
Before we formulate our variational problem, we prove a general result for the continuity equation (the second equation in \eqref{eq:MFG system}) that motivates the choice of the function spaces. We recall the following result from \cite{CARAVENNA20161168} about the existence and uniqueness of solutions to the continuity equation. Let $\mu_0\in C^1(\Rr)$, $b\in L^1([0,T];W_{loc}^{1,1}(\Rr))\cap C([0,T]\times\Rr)$ and $b_x\in L^\infty([0,T]\times\Rr)$. Then, the continuity equation  
\begin{equation}\label{cont-eq}
\begin{cases}
\mu_t-(b\mu)_x=0 & [0,T]\times \Rr,
\\
\mu(0,x)=\mu_0(x) & x \in\Rr
\end{cases}
\end{equation}
has  a unique solution $\mu\in L^\infty([0,T]\times\Rr)$  in distributional sense. The existence result follows from Theorem 1.1 in \cite{CARAVENNA20161168}, which addresses the existence and uniqueness of distributional solutions to the continuity equation \eqref{cont-eq} for a vector field $b$ satisfying weaker conditions.

Now, we prove that if the initial condition $\mu_0$ of the continuity equation is compactly supported, the solution $\mu$ is also compactly supported.

\begin{pro}\label{cont-comp}
Let $\mu_0\in C^1(\Rr)$, $b\in L^1([0,T];W_{loc}^{1,1}(\Rr))\cap C([0,T]\times\Rr)$ and $b_x\in L^\infty([0,T]\times\Rr)$. Assume further that $\mu_0 \in C_c^1(\Rr)$. 
Then, the unique solution to the continuity equation \eqref{cont-eq} has compact support; that is,  $\mu \in L_c^\infty([0,T]\times\Rr)$.
\end{pro}

\begin{proof} 
From the results in \cite{CARAVENNA20161168}, it follows that there exists a unique, $\mu\in L^\infty([0,T]\times\Rr)$ solving \eqref{cont-eq} in the distributional sense.  
Let $b^\varepsilon$ be a sequence of functions in $C^\infty([0,T]\times \Rr)$ satisfying:
\begin{itemize}
	\item $b^\varepsilon$ is Lipschitz continuous w.r.t. $x$, and its Lipschitz constant satisfies $\mbox{Lip}(b^\varepsilon) \leq \mbox{Lip}(b)$,
	\item $b^\varepsilon \to b$ uniformly on every compact set of $[0,T]\times \Rr$.
\end{itemize}
We can obtain such sequence $b^\varepsilon$ by considering the convolution with standard mollifiers in $x$ and a partition of unity construction in $t$.
Next, we consider the continuity equation with the vector field $b^\varepsilon$
	\begin{equation}\label{cont-eq-reg}
	\begin{cases}
		\mu^\varepsilon_t-(b^\varepsilon\mu^\varepsilon)_x=0 & [0,T]\times \Rr,
		\\
		\mu^\varepsilon(0,x)=\mu_0(x) & x\in\Rr.
	\end{cases}
\end{equation}
Because $b^\varepsilon,\, b_x^\varepsilon\in C^1([0,T]\times\Rr)$, by Theorem 6.3 in \cite{perthame}, \eqref{cont-eq-reg} has a unique solution $\mu^\varepsilon\in C^1([0,T]\times\Rr)$ given by
\begin{equation}\label{sol-smooth-rep}
\mu^\varepsilon(t,X^\varepsilon(t;y))J(t;y)=\mu_0(y), \quad (t,y)\in[0,T]\times\Rr,
\end{equation}
where 
\begin{equation}\label{def-J}
J(t;y)=\exp\left( \int_{0}^{t}b^\varepsilon_x(s,X^\varepsilon(s;y))~\ds \right),
\end{equation}
and $X^\varepsilon$ solves the following initial value problem
\begin{equation}\label{eq:Characteristics equation}
\begin{cases}
\dot{X}^\varepsilon(t;y)=b^\varepsilon(t,X^\varepsilon(t;y))&\quad (t,y)\in(0,T]\times \Rr,
\\
X^\varepsilon(0;y)=y &\quad y\in \Rr.
\end{cases}
\end{equation}
Because $b^\varepsilon\in C^1([0,T]\times\Rr)$, the map $y \mapsto X^\varepsilon(t;y)$ is a diffeomorphism (see \cite{Perko2006}, Chapter 3). Moreover, because $b^\varepsilon \in C^1([0,T]\times\Rr)$ is Lipschitz continuous w.r.t. $x$, we have $|b^\varepsilon(t,x)| \leq C_\varepsilon\left(1+|x|\right)$, where, by the uniform convergence of $b^\varepsilon$ to $b$ on compact sets,
\begin{equation}\label{eq:Gronwall bound}
	C_{\varepsilon}\leq \max\{1+|b(t,0)|,\mbox{Lip}(b)\}.
\end{equation}
Applying Gr\"{o}nwall's inequality to \eqref{eq:Characteristics equation}, provides
\begin{equation}\label{diffeom-ineq}
|X^\varepsilon(t;y)|\leq \left( |y| + C_\varepsilon T\right)\left(1+C_\varepsilon Te^{C_\varepsilon T}\right).
\end{equation}
Because $|y|>R_0$ implies $\mu_0(y)=0$ for some $R_0>0$, \eqref{sol-smooth-rep} shows that $\mu^\varepsilon$ may have non-zero values only for those $y$ satisfying $|y|\leq R_0$, for which \eqref{diffeom-ineq} implies 
\[
	|X^\varepsilon(t;y)|\leq \left(R_0 + C_\varepsilon T\right)\left(1+C_\varepsilon Te^{C_\varepsilon T}\right)=:R_{\varepsilon}.
\]
Thus, $\mbox{supp}(\mu^\varepsilon) \subset [0,T]\times [-R_\varepsilon,R_\varepsilon]$, which, by \eqref{eq:Gronwall bound}, provides the existence of $R>0$, depending on $T$ and $\mbox{Lip}(b)$, such that $\supp\mu^\varepsilon\subset[-R,R]$ for every $1\gg \varepsilon>0$.
 Furthermore, \eqref{sol-smooth-rep} and \eqref{def-J} imply that there exists $C\geq0$, depending on $T$, $\mbox{Lip}(b)$ and $\mu_0$, such that 
$||\mu^\varepsilon||_{L^\infty([0,T]\times\Rr)}\leq C$ for all $1\gg \varepsilon>0$. 
Hence, by Banach-Alaoglu theorem, there exists $\bar{\mu}\in L^\infty([0,T]\times\Rr)$ such that
\begin{equation*}
\mu^\varepsilon\overset{\ast}{\rightharpoonup} \bar{\mu}  \quad \mbox{as} \quad  \varepsilon\to 0 \quad \text{in}\quad   L^\infty([0,T]\times \Rr).
\end{equation*}
Consequently,  $\supp\bar{\mu}\subset[-R,R]$ as well. On the other hand, $\mu^\varepsilon$ also solves \eqref{cont-eq-reg} in the sense of distributions; that is,
\begin{equation}\label{sol-weak-reg}
-\int_{0}^{T}\int_{\Rr}\mu^\varepsilon\left(\phi_t-b^\varepsilon\phi_x\right)\dx\dt=\int_{\Rr}\mu_0\phi~\dx, 
\end{equation}
for any $\phi\in C^1_c([0,T)\times\Rr)$.

Thus, given $\phi\in C^1_c([0,T)\times\Rr)$, we write
\begin{equation}\label{eq:Aux weak convergence}
\begin{split}
&-\int_0^T \int_{\Rr} \bar{\mu} \left(\phi_t - b\phi_x\right) \dx \dt 
\\
& = -\int_{0}^{T}\int_{\Rr}\mu^\varepsilon\left(\phi_t-b^\varepsilon\phi_x\right) + \left(\mu^\varepsilon - \bar{\mu}\right) \left(\phi_t - b \phi_x\right) +  \mu^\varepsilon \phi_x \left(b-b^\varepsilon\right) \dx \dt.
\end{split}
\end{equation}
Because $\phi_t$, $\phi_x$, $b\phi_x \in L^1([0,T]\times \Rr)$, the second term on the right-hand side of \eqref{eq:Aux weak convergence} vanishes as $\varepsilon \to 0$. 
 Furthermore, using the uniform bound for $\mu^\varepsilon$, and because $b^\varepsilon$ converges uniformly to $b$ in the compact support of $\phi$, we obtain that the third term on the right-hand side of \eqref{eq:Aux weak convergence} also vanishes as $\varepsilon \to 0$. 
Thus, using \eqref{sol-weak-reg}, we get
\[
-\int_{0}^{T}\int_{\Rr}\bar{\mu}\left(\phi_t-b\phi_x\right)\dx\dt=\int_{\Rr}\mu_0\phi~\dx, 
\]
and since $\phi$ is arbitrary, we conclude that $\bar{\mu}$ is a solution to \eqref{cont-eq} in the distributional  sense.

To conclude the proof, it is enough to recall that the results in \cite{CARAVENNA20161168} provide uniqueness for the initial value problem in \eqref{cont-eq}, in the sense of the distributions.
\end{proof}

Applying the previous result to the MFG system \eqref{eq:MFG system}-\eqref{boundaryP}, we obtain the following.

\begin{cor}\label{corollary-m-comp}
Suppose that Assumptions \ref{hyp: grow-L}-\ref{hyp: V-uT convex DS}  hold. Let $(u,m,\varpi)$ be the solution to \eqref{eq:MFG system}. 
Assume further that  Assumption \ref{hyp: m_0-comp} holds with $R_0>0$.
Then, $m$ is compactly supported; that is, there exists a constant $R_m\geq R_0$, such that $\supp m(t,\cdot)\subseteq[-R_m,R_m]$ for $t\in[0,T]$. Moreover, $R_m$ is bounded by a constant that depends only on the problem data.
\end{cor}
\begin{proof}
Let $b(t,x)=H'(\varpi(t)+u_x(t,x))$ denote the vector field of the continuity equation in \eqref{eq:MFG system}. By Proposition 8 in \cite{gomes2018mean}, $|u_{xx}|\leq C(T,V,u_T)$, which implies that $u_x$ is Lipschitz w.r.t. $x$. 
By Assumptions \ref{hyp: grow-L} and \ref{hyp: H convex}, for $p>2$, $|H''|\leq C(F)$, for some $C(F)>0$. Thus, $b(t,\cdot)\in C^1(\Rr)$ is Lipschitz continuous in $\Rr$ uniformly with respect to $t$. Furthermore, the Lipchitz constant satisfies 
\[
	\mbox{Lip}(b(t,\cdot)) \leq C_0,
\]
where $C_0=C_0(T,V,F,u_T)$. Therefore, Proposition \ref{cont-comp} implies the first part of the result. Moreover, \eqref{diffeom-ineq} shows that 
\begin{equation}\label{eq:Compact support bound depending on data}
R_m \leq \left(R_0+C_0 T\right)\left(1+C_0 T e^{C_0 T}\right),
\end{equation}
which concludes the proof.
\end{proof}

\begin{remark}\label{comp-rem}
Consider the potential $\varphi$ associated with the solution $(u,m,\varpi)$ of the MFG system \eqref{eq:MFG system}-\eqref{boundaryP}, as given by \eqref{eq: potential relations}. By Corollary \ref{corollary-m-comp}, we deduce that the gradient of the potential $\varphi$ has compact support; that is, $\supp(\varphi_t(t,\cdot)),\supp(\varphi_x(t,\cdot))\subseteq[-R_m,R_m]$ for all $t\in[0,T]$. Thus, \eqref{eq:Compact support bound depending on data} shows that, by selecting
\[
	R > \left(R_0+C_0 T\right)\left(1+C_0 T e^{C_0 T}\right),
\]
we obtain a compact set $[-R,R]$ that depends only on problem data, and which contains the support of the gradient of $\varphi$ when \eqref{eq:Potential in terms of MFGs} holds. Thus, using this compact set, we can formulate our variational problem independently of the solution $(u,m,\varpi)$ of the MFGs system \eqref{eq:MFG system} and \eqref{boundaryP}. Notice that \eqref{eq:Potential in terms of MFGs} already suggests a candidate for a minimizer. However, we study the existence of solutions to the variational problem independently of solutions to the MFGs system. Moreover, if uniqueness holds and we have existence for both problems, then \eqref{eq:Potential in terms of MFGs} is the unique minimizer.

\end{remark}

\subsection{Statement of the variational problem}
In this subsection, we present our variational approach rigorously using only problem data. We start with the notations and the definition of admissible functions. Then, we formulate and study the variational problem.

Let $R_0$ be given by Assumption \ref{hyp: m_0-comp} and let 
\begin{equation}\label{eq:R selection}
R>\max\left\{ \left(R_0+C_0 T\right)\left(1+C_0 T e^{C_0 T}\right) , R_0+\|Q\|_{L^1([0,T])}\right\}.
\end{equation}
Notice that, by \eqref{eq:Compact support bound depending on data}, $R$ is an upper bound for $R_m$, as required, according to Remark \ref{comp-rem}. The additional requirement  $R > R_0+\|Q\|_{L^1([0,T])}$ guarantees that the set of admissible functions that we define below is not empty. Set
\[
	\Omega_R=[0,T]\times[-R,R], \quad \Omega=[0,T]\times\Rr.
\] 
We denote by $\mathcal{M}(\Omega_R)$ ($\mathcal{M}(\Omega)$)  the set of Radon measures on $\Omega_R\subset\Rr^2$ ($\Omega\subset\Rr^2$) and by $BV(\Omega_R)$ ($BV(\Omega)$) the set of functions with bounded variation on $\Omega_R$ ($\Omega$) (see \cite{evansgariepy2015}, \cite{AFP2000}).
%

To  define the admissible set for our variational problem, we rewrite the balance condition, the second equation in \eqref{eq: Euler-Lagrange wrt potential}. Recall that $\supp(m_0) \subset [-R_0,R_0]$ and $R> R_0$.
Let 
\begin{equation}\label{def-M0}
	M_0(x) = \int_{-\infty}^x m_0(y) \dy = \int_{-R}^x m_0(y) \dy, \quad x \in \Rr,
\end{equation}
be the cumulative density function of $m_0$. Note that after integrating the balance condition over $[0,t]$, and requiring that $\int_{\Rr} \varphi_x (t,x) \dx = 1$ for $t \in [0,T]$ (which follows in case that \eqref{eq: potential relations} holds), we get 
\begin{equation*}
	\int_{0}^{t}\int_{\Rr}\varphi_t~\dx \ds=-\int_{0}^{t}Q(s) \ds, \quad t \in [0,T].
\end{equation*}
Therefore, we  write the balance condition as 
\begin{equation}\label{for-Poincare}
	\int_{\Rr} \varphi(t,x)-M_0(x)~\dx=-\int_{0}^{t}Q(s) \ds, \quad t \in [0,T].
\end{equation}
Relying on \eqref{for-Poincare} and taking into account the discussion in Remark \ref{comp-rem}, for any set $A\subset \Rr^2$ satisfying $[0,T]\times[-R,R]\subseteq A$,  we denote  
\begin{align*}
\mathcal{B}_R(A)=& \left\{ (\varphi -M_0) \in W^{1,1}(A):~ \supp(\varphi_t(t,\cdot)),\supp(\varphi_x(t,\cdot))\subseteq (-R,R),~ t\in [0,T]\right\},
\\
\mathcal{B}(A)= & \left\{ \varphi \in \mathcal{B}_R(A):~\varphi_x\geq 0,~\varphi(0,x)=\int_{-R}^x m_0(y)\dy, ~  x \in \Rr \right.
	\\
	& \quad \left. \int_{\Rr} \varphi(t,x)-M_0(x)~\dx=-\int_{0}^{t}Q(s) ~\ds, ~ \int_{-R}^{R} \varphi_x(t,x)\dx=1, ~  t \in [0,T]	\right\},
\end{align*} 
which are convex sets. Before proceeding, we prove a crucial property of the set $\mathcal{B}(\Omega)$.

\begin{proposition}\label{pro-phi-0-1}
For any function $\varphi\in\mathcal{B}(\Omega)$, we have, for $t\in[0,T]$,
	\begin{equation*}
\varphi(t,x)=\begin{cases}
	0\quad & x\in(-\infty,-R],
	\\
\varphi(t,x)\quad & x\in(-R,R),
	\\
1 \quad & x\in[R,+\infty).
\end{cases}
	\end{equation*}
\end{proposition}
\begin{proof}
Because $(\varphi -M_0) \in W^{1,1}(\Omega)$ and $\lim\limits_{x\to+\infty}M_0(x)=1$, for each $t \in [0,T]$, there exists a sequence ${x_k}$ such that $x_k\to\infty$ and  $\lim\limits_{k\to+\infty}\varphi (t, x_k)=1$. On the other hand, recalling that  
\[
	\supp(\varphi_t(t,\cdot)),\supp(\varphi_x(t,\cdot))\subset(-R,R),
\]
we have that $\varphi$ is constant on $\Omega\setminus [0,T]\times(-R,R)$. Consequently, $\varphi(t,x)=1$ for $x\in[R,+\infty)$. Similarly, we can prove that $\varphi(t,x)=0$, $x\in(-\infty,-R]$.
\end{proof}

Finally, the set of admissible functions for our variational problem is given by 
\begin{equation}\label{def:admissible-A}
			\begin{split}
				\mathcal{A}(\Omega_R)=\left\{ \varphi \in \mathcal{B}(\Omega_R):
				~\varphi(t,-R)=0
				\right\}.
			\end{split}
\end{equation} 
As a result of Proposition \ref{pro-phi-0-1}, we obtain the following relation between the admissible set $\mathcal{A}(\Omega_R)$ and the set $\mathcal{B}(\Omega)$. 
\begin{corollary}\label{cor-AU-to-AO} For any function $\varphi\in \mathcal{B}(\Omega)$ there exist a function $\tilde{\varphi}\in\mathcal{A}(\Omega_R)$ such that 
$\varphi\equiv\tilde{\varphi}$ in $\Omega_R$. The opposite is also true. 
\end{corollary}
Under Assumption \ref{hyp: V-uT Lipschitz 2nd D bounded DS}, we have
\begin{equation}\label{eq-need-for-pro-4.7}
\left| \int_{\Omega_R}u^\prime_T(x) \varphi_t ~\dx\dt\right| \leq \mbox{Lip}{(u_T)}\int_{\Omega_R}|\varphi_t|  ~\dx\dt,
\end{equation}
where $\mbox{Lip}{(u_T)}$ is the Lipschitz constant of $u_T$. Relying on the previous inequality, we consider the following variational problem 
\begin{equation*}
\begin{split}
\inf_{\varphi\in\mathcal{B}(\Omega)} \int_{\Omega_R} L(\varphi_t,\varphi_x) -V\varphi_x-u^\prime_T(x) \varphi_t~\dx\dt,
\end{split}
\end{equation*}
which, by Corollary \ref{cor-AU-to-AO}, coincides with the following (see \eqref{functional-v})
\begin{equation}\label{var-problem-c}
		\inf_{\varphi\in\mathcal{A}(\Omega_R)} I[\varphi],
\end{equation}
where
\begin{equation*}
	I[\varphi] :=\int_{\Omega_R} L(\varphi_t,\varphi_x) -V\varphi_x -u^\prime_T(x) \varphi_t~\dx\dt.
\end{equation*}
As anticipated in Remark \ref{comp-rem}, \eqref{eq:R selection} guarantees that the previous variational problem does not rely on the solution $(u,m,\varpi)$ to \eqref{eq:MFG system}-\eqref{boundaryP} but only on the data of Problem \ref{PMFG}.
Moreover, the infimum in \eqref{var-problem-c} can be attained by at most one function, as we show next.
	
\begin{pro}\label{pro-unique}  
Suppose that Assumptions \ref{hyp: grow-L}-\ref{hyp: m_0-comp}  hold. Then, at most, one function attains the infimum in \eqref{var-problem-c}.
\end{pro}
\begin{proof}
Let $\varphi^1$ and $\varphi^2$ attain the infimum in \eqref{var-problem-c}.
By Proposition \ref{pro-bound-L1}, we denote 
\begin{equation*}
\begin{aligned}
\ell=\min\limits_{(\varphi)\in\mathcal{A}(\Omega_R)}I[\varphi]\in\Rr.
\end{aligned}
\end{equation*}
Thus, $I[\varphi^1]=I[\varphi^2]=\mathcal{\ell}$. Setting $\bar{\varphi}=\frac{1}{2}(\varphi^1+\varphi^2)$, and using  the convexity of $L$, we obtain
\begin{equation}\label{eq:mincxty0}
\ell\leq I[\bar{\varphi}]\leq\tfrac{1}{2}I[\varphi^1]+\tfrac{1}{2}I[\varphi^2]=\ell.
\end{equation}
Hence, $\bar{\varphi}$ is also minimizer of \eqref{var-problem-c}.
Let $\tilde{\varphi}=\frac{\varphi^1+\bar{\varphi}}{2}$ and
\begin{equation*}
\begin{aligned}
\Uu_1=\{(t,x)\in \Omega_R:{\varphi}_x^1>0 \},\quad \Uu_2=\{(t,x)\in \Omega_R:{\varphi}_x^2>0 \},
\\
\tilde{\Uu}=\{(t,x)\in \Omega_R:\tilde{\varphi}_x> 0 \}= \bar{\Uu}= \{(t,x)\in \Omega_R:\bar{\varphi}_x> 0 \}=\Uu_1\cup\Uu_2.
\end{aligned}
\end{equation*}
Arguing as in \eqref{eq:mincxty0}, we have
\begin{equation}\label{eq:mincxty}
\ell\leq I[\tilde{\varphi}]\leq\tfrac{1}{2}I[\varphi^1]+\tfrac{1}{2}I[\bar{\varphi}]=\ell.
\end{equation}
This with \eqref{eq:mincxty0}, yields that
\begin{equation*}
\begin{aligned}
L(\varphi^1_t,\varphi^1_x),\,L(\varphi^2_t,\varphi^2_x),\,L(\bar{\varphi}_t,\bar{\varphi}_x), \, L(\tilde{\varphi}_t,\tilde{\varphi}_x) <+\infty\enspace \hbox{a.e.~in } \Omega_R.
\end{aligned}
\end{equation*}
Therefore,  
\begin{equation}
\label{eq:defLL0}
\begin{aligned}
&\varphi^1_t=0 \enspace\text{ a.e.~in } \Omega_R\setminus\Uu_1,\\
&\varphi^2_t=0 \enspace\text{ a.e.~in } \Omega_R\setminus\Uu_2,\\
& \bar\varphi_t =0 \enspace\text{ a.e.~in }\Omega_R\setminus\bar{\Uu}\\
& \tilde{\varphi}_t =0 \enspace\text{ a.e.~in } \Omega_R\setminus\tilde{\Uu}. 
\end{aligned}
\end{equation}
Furthermore,
\eqref{eq:mincxty} implies          
\begin{equation}\label{sumalsomin}
\begin{split}
\int_{\Omega_R}\Big(\tfrac{1}{2}L(\varphi^1_t,\varphi^1_x)+&\tfrac{1}{2}L(\bar{\varphi}_t,\bar{\varphi}_x)-L(\tilde{\varphi}_t,\tilde{\varphi}_x)\Big)\,\dx\dt=0.
\end{split}
\end{equation}
The convexity of $L$ and \eqref{sumalsomin} implies                
\begin{equation*}
\tfrac{1}{2}L(\varphi^1_t,\varphi^1_x)+\tfrac{1}{2}L(\bar\varphi_t,\bar\varphi_x)-L(\tilde{\varphi}_t,\tilde{\varphi}_x)=0 ,\quad \text{a.e. in } \Omega_R.  
\end{equation*}
Consequently, the following also holds
\begin{equation}\label{eq:Fcxcomb=0}
\tfrac{1}{2}L(\varphi^1_t,\varphi^1_x)+\tfrac{1}{2}L(\bar\varphi_t,\bar\varphi_x)-L(\tilde{\varphi}_t,\tilde{\varphi}_x)=0 ,\quad \text{a.e. in }\quad \Uu_1\cap\bar{\Uu}\cap\tilde{\Uu}.  
\end{equation}
Because $L$ is strictly convex in $\Rr\times \Rr^+$ and $\Uu_1\cap\bar{\Uu}\cap\tilde{\Uu}=\Uu_1\subset\Rr\times \Rr^+$, we obtain from \eqref{eq:Fcxcomb=0} that
\begin{equation*}
\begin{cases}
\varphi^1_t=\bar\varphi_t
\\
\varphi^1_x=\bar\varphi_x.
\end{cases}\quad\text{a.e.  in}\quad \Uu_1 
\end{equation*}
Hence, 
\begin{equation}\label{eq:equality-in-U1}
\begin{cases}
\varphi^1_t=	\varphi^2_t
\\ 
\varphi^1_x=	\varphi^2_x.
\end{cases} \quad\text{a.e.  in}\quad \Uu_1
\end{equation}
Taking $\varphi^2$ instead of $\varphi^1$ in \eqref{eq:mincxty} and arguing as before, we obtain
\begin{equation}\label{eq:equality-in-U2}
\begin{cases}
\varphi^1_t=	\varphi^2_t
\\ 
\varphi^1_x=	\varphi^2_x.
\end{cases} \quad\text{a.e.  in}\quad \Uu_2
\end{equation}
Combing  \eqref{eq:defLL0}, \eqref{eq:equality-in-U1} and \eqref{eq:equality-in-U2}, we conclude that $\varphi^1=\varphi^2$.
\end{proof}

Next, we prove that the infimum in \eqref{var-problem-c} is bounded.
\begin{proposition}\label{pro-bound-L1}
Assume that Assumptions \ref{hyp: grow-L}-\ref{hyp: m_0-comp}  hold. 
Then, there exist positive constants, $C_1$ and $C_2$, depending only on the problem data such that
\begin{equation}
	\label{inf-bound}
-C_2\leq\inf_{\varphi \in \mathcal{A}(\Omega_R)} I[\varphi]  \leq C_1.
\end{equation}
Furthermore, there exists a positive constant, $C$, depending only on problem data, such that  for every minimizing sequence $\{\varphi^n\}_{n\in\Nn}$ of the variational problem \eqref{var-problem-c}, we have
	\begin{equation}\label{prop-L1-bound-comp}
		\int_{\Omega_R} \tfrac{\left|\varphi_t^n\right|^p}{(\varphi^n_x)^{p-1}}\dx \dt,\quad	\|\varphi^n_t\|_{L^1(\Omega_R)} \leq C.
	\end{equation}
\end{proposition}

\begin{proof} 
First, we prove the upper bound in \eqref{inf-bound}.
Let
\begin{equation*}
	\varphi^0(t,x)=M_0(x-q(t)), 
\end{equation*}
where $M_0$ is defined by \eqref{def-M0} and  $q(t)=\int_{0}^{t}Q(\tau)~d\tau$. 
Therefore, since $q(0)=0$, we have
\begin{equation}\label{phi001}
 \varphi^0(0,x)=\int_{-R}^{x} m_0(y)\dy,\quad	\varphi^0_x=m_0(x-q(t)),\quad \varphi^0_t=-m_0(x-q(t))Q(t).
\end{equation}
Thus, 
\begin{equation}\label{eq:A(Omega_R) not empty}
\varphi^0\in\mathcal{A}(\Omega_R).
\end{equation} 
Taking into account \eqref{phi001}, we have
\begin{align}\label{up-bound}
\inf_{\mathcal{A}(\Omega)} I[\varphi]\leq	I[\varphi^0] & \leq \int_{\Omega_R}L(\varphi^0_t,\varphi^0_x) +||V||_{L^\infty([-R,R])}\varphi_x^0 +||u_T^\prime||_{L^\infty(\Omega_R)}|\varphi_t^0|~\dx\dt =: C_1 .
\end{align}
Next, relying on this bound, we prove \eqref{prop-L1-bound-comp}, implying the lower bound in \eqref{inf-bound}. 
By Assumption \ref{hyp: grow-L} and \ref{eq: Legendre tr}, $ L\geq 0$. Thus, for all $\varphi \in \mathcal{A}(\Omega_R)$, we have
\begin{equation}\label{first}
-\int_{\Omega_R}V\varphi_x+u^\prime_T \varphi_t\dx\dt \leq I[\varphi].
\end{equation}
Recalling that $V$ is continuous and taking into account  Assumption \ref{hyp: V-uT Lipschitz 2nd D bounded DS} by  \eqref{eq-need-for-pro-4.7} and  \eqref{first}, we get
\begin{equation}\label{for-lower-bound}
		-C(u_T,R,V)- \mbox{Lip}(u_T)  \int_{\Omega_R}\left| \varphi_t (t,x) \right|\dx \dt \leq I[\varphi].
\end{equation}
From \eqref{up-bound} follows that for any minimizing sequence $\{\varphi^n\}_{n\in\Nn}$, there exists $N$ such that $n \geq N$ implies $I[\varphi^n]\leq C_1+1$. 
Consequently, recalling the definition of $L$ in \eqref{L def}, by Assumption \ref{hyp: grow-L} and \eqref{eq-need-for-pro-4.7}, we deduce that 
\begin{equation}\label{phi_t-over-phix}
\begin{split}
c	\int_{\Omega_R} \frac{\left|\varphi_t^n\right|^p}{(\varphi^n_x)^{p-1}}\dx \dt \leq& C_1+1+||V||_{L^\infty([-R,R])}\int_{\Omega_R}\varphi^n_x\dx\dt+ \mbox{Lip}(u_T) \int_{\Omega_R}  \left|\varphi^n_t(t,x)\right| \dx \dt\\ \leq&C+\mbox{Lip}(u_T) \int_{\Omega_R}  \left|\varphi^n_t(t,x)\right| \dx \dt,
\end{split}
\end{equation}
for all $n \in \Nn$. On the other hand, by Young's inequality, we have
\begin{equation}\label{eq-phi_t-bound}
\int_{\Omega_R} \left| \varphi_t^n \right| \dx \dt =  \int_{\Omega_R} \frac{\left| \varphi_t^n \right|}{( \varphi^n_x)^{\frac{p-1}{p}}} ( \varphi^n_x)^{\frac{p-1}{p}}\dx \dt \leq \varepsilon \int_{\Omega_R}  \frac{\left|\varphi_t^n\right|^p}{ (\varphi^n_x)^{p-1}}\dx \dt + C(\varepsilon) \int_{\Omega_R} \varphi_x^n\dx \dt,
\end{equation}
 where  $\varepsilon=\frac{c}{2\mbox{Lip}(u_T)}$. 
Recalling that $\int_{-R}^{R} \varphi_x(\cdot,x)\dx=1$, the preceding inequality and \eqref{phi_t-over-phix} imply \eqref{prop-L1-bound-comp}.  
Finally, \eqref{prop-L1-bound-comp} and \eqref{for-lower-bound} yield the lower bound in \eqref{inf-bound}. 
\end{proof}

Thus, for all minimizing sequences of the variational problem \eqref{var-problem-c}, we obtain uniform bounds in $W^{1,1}(\Omega_R)$. Therefore, any minimizing sequence has a weakly convergent sub-sequence in  $BV(\Omega_R)$ (\cite{evansgariepy2015}, Chapter 5). 
However, it is not guaranteed that the infimum in \eqref{var-problem-c}  is attained in $\Aa(\Omega_R)$. Therefore, we enlarge the set of admissible functions by relaxing the conditions defining $\mathcal{A}(\Omega_R)$, as we present in the next section.

\subsection{Relaxed variational problem}
Here, we relax the variational problem \eqref{var-problem-c} to ensure the existence of minimizers in the set of admissible functions.

First, we extend the functional in \eqref{var-problem-c} to the convex set
\[
	BV_0^+(\Omega_R)=\{\psi\in BV(\Omega_R):\psi_x\geq 0\}.
\]
For that, let $W:BV_0^+(\Omega_R) \to \Rr \cup \{+\infty\}$ be given by
\begin{equation*}
W[\varphi]=\begin{cases}
\int_{\Omega_R} L(\varphi_t,\varphi_x) -V\varphi_x-u^\prime_T(x) \varphi_t~\dx\dt,~ \varphi\in W^{1,1}(\Omega_R)\cap BV_0^+(\Omega_R)\\
+\infty, \quad \mbox{otherwise}.
  \end{cases}
\end{equation*}
In $BV(\Omega_R)$ we consider the intermediate convergence; that is, $(\varphi_k)_{k\in\Nn}\subset BV(\Omega_R)$ converges  to $\varphi \in BV(\Omega_R)$ in the intermediate (or strict) sense if 
\[
	\varphi_k \to \varphi \mbox{ in } L^1(\Omega_R)\quad \mbox{and} \quad \|D\varphi_k\|(\Omega_R) \to \|D\varphi\|(\Omega_R),
\]
where $\|D\varphi\|(\Omega_R)$ is the total variation of the measure $D\varphi$ on $\Omega_R$ (see \cite{AFP2000}). We recall that  $W^{1,1}(\Omega_R)$ is dense in $BV(\Omega_R)$ with respect to the intermediate convergence (see Theorem 10.1.2 in \cite{attouch2014variational}). We aim to define a functional $\Ww$, the sequential lower semicontinuous envelope of $W$ w.r.t. intermediate convergence on $BV$ (Chapter 3, \cite{DalMasoGammaIntro}); that is
\begin{align*}
	\Ww[\varphi]=&\sup\left\{G[\varphi]:~ G\leq W,~ G \mbox{ is sequentially lower semicontinuous on } BV(\Omega_R) \right.
	\\
	& \quad \quad \left. \mbox{ w.r.t. intermediate convergence}\right\},
\end{align*}
which is the greatest functional below $W$ that is sequentially lower semi-continuous w.r.t intermediate convergence in $BV(\Omega_R)$. Let
\begin{align*}
\mathcal{J}[\varphi]=&\inf\Biggl\{\Biggr. \liminf\limits_{n\to\infty} I[\varphi^n]:\, \{\varphi^n\}\subset W^{1,1}(\Omega_R)\cap BV_0^+(\Omega_R),
\\ 
&  \quad \quad \quad  \varphi^n\to\varphi\text{ in the sense of the intermediate convergence in } BV(\Omega_R) \Biggl.\Biggr\}. 
\end{align*}
Next, we prove that actually $\Ww=\mathcal{J}$ and obtain explicit expression for $\mathcal{J}$. 

Assuming that Assumption \ref{hyp: grow-L} holds for some $c>0$, $p>1$ and arguing as in \eqref{eq-phi_t-bound} by using Young's inequality, we obtain
\begin{equation}\label{eq-lower-bound-for-L}
	cp(|v_1|+v_2)\leq L(v_1,v_2)+c(2p-1)v_2, \quad (v_1,v_2)\in\Rr\times\Rr^+_0.
\end{equation}
Let
\begin{equation}\label{def-f}
	f(v_1,v_2)=L(v_1,v_2)-u^\prime_Tv_1-Vv_2=f_{N}(v_1,v_2)+f_{L}(v_1,v_2),
\end{equation}
where $f_{N}(v_1,v_2)=L(v_1,v_2)+c(2p-1)v_2$ and $f_{L}(v_1,v_2)=-u^\prime_Tv_1-(V+c(2p-1))v_2$.
According to Lemmas 8.1 and 8.3 in \cite{DRT2021Potential}, if Assumption \ref{hyp: H convex} holds, the function $L$ defined in \eqref{L def} is convex and lower semicontinuous  in $\Rr\times\Rr_0^+$, therefore, $f_{N}$ as well. 

%

For the next result,  we compute the  recession function, $\bar{f}_N$, of $f_N$, which is given by 
\[
	\bar{f}_N(z,y):= \sup\left\{ f(w_1+z,w_2+y)-f(w_1,w_2):~ (w_1,w_2)\in \mbox{dom}_e (f_N )\right\}, \quad (z,y) \in \Rr\times\Rr,
\]
where 
\[
	\mbox{dom}_e(f_N):= \left\{ (w_1,w_2) \in \Rr \times \Rr_0^+:~ f_N(w_1,w_2) < +\infty \right\}.
\] 
Because $f_N$ is convex, from Theorem 4.70 in \cite{FoLe07}, we have
\begin{equation*}
	\bar{f}_N(z,y) = \lim\limits_{t\to\infty}\frac{f_N((z,y)t+(w_1,w_2))-f_N(w_1,w_2)}{t},
\end{equation*}
for any $(w_1,w_2)\in\Rr\times\Rr_0$.
Taking $(w_1,w_2)=(0,0)$ in the preceding equation and considering \eqref{L def}, we deduce that $f_N$ is equal to its  recession function; that is,
\begin{equation*}
	\bar{f}_N(z,y)=f_N(z,y) = \begin{cases} F\left( \frac{z}{y}\right) y+c(2p-1)y & (z,y) \in \Rr\times\Rr^+,
	\\
+\infty & z\neq 0, y=0,
	\\
 0& z=0, y= 0,\end{cases}
\end{equation*}
where constants $c$ and $p$ are given by Assumption \ref{hyp: grow-L}.
Using the preceding observation, we  prove that the first integrant of the functional in \eqref{functional-v} is sequentially lower semi-continuous w.r.t. the weak -$*$ convergence of measures.

\begin{pro}\label{pro-semi-lower1} 
Suppose that Assumption \ref{hyp: H convex} holds. 
Let $(v_1^n,v_2^n)\in L^1(\Omega_R)\times L^1(\Omega_R;\Rr_0^+)$ and $\mu=(\mu_1,\mu_2)\in\Mm(\Omega_R)\times\Mm(\Omega_R;\Rr_0^+)$ be such that 
\begin{equation*}
	(v_1^n,v_2^n)\Ll^2\lfloor \Omega_R\overset{\ast}{\rightharpoonup}(\mu_1,\mu_2),\quad\text{ in }  \Mm(\Omega_R)\times\Mm(\Omega_R;\Rr_0^+).
\end{equation*}
Then,
\begin{align*}
&\liminf\limits_{n\to\infty}\int_{\Omega_R}f_N(v_1^n,v_2^n)\dx\dt
\\
&\geq \int_{\Omega_R}f_N\left( \frac{d\mu}{d \Ll^2}(t,x)\right) \dx\dt+ \int_{\Omega_R}f_N\left( \frac{d\mu_s}{d ||\mu_s||}(t,x)\right) d||\mu_s||(t,x),
\end{align*}
where $\mu=\frac{d\mu}{d \Ll^2}\lfloor\Omega_R+\mu_s$ is the Radon-Nikodym decomposition of $\mu$ and $\frac{d\mu_s}{d ||\mu_s||}$ is  the Radon-Nikodym derivative of $\mu_s$ with respect to  its total variation.
\end{pro}

\begin{proof} 
By Assumption \ref{hyp: H convex}, $f_N$ is convex and lower semi-continuous in $\Rr\times\Rr_0^+$ and $\bar{f}_N=f_N$. Hence, the proof follows from Theorem 5.19 in \cite{FoLe07}.
\end{proof}

\begin{pro}\label{pro-semi-lower2} 
Suppose that  Assumption \ref{hyp: V-uT Lipschitz 2nd D bounded DS} holds. Let $v_1^n\to v_1$ and $v_2^n\to v_2$ weakly  in $\Mm(\Omega_R)$.
Then, 
\begin{equation*}
\lim\limits_{n\to\infty}\left( -\int_{\Omega_R}(V+C)v_1^n +u_T^\prime v^n_2\dx\dt \right) = -\int_{\Omega_R}(V+C)v_1  +u^\prime_T(x) v_2 ~\dx dt,
\end{equation*}
for any $C\in\Rr$.
\end{pro}

\begin{proof}
It is enough to notice that the functions $V$ and $u'_T(x)$ are continuous.
\end{proof}
Now, we are ready to prove that $\Ww=\mathcal{J}$.
\begin{theorem}\label{theorem-lower} Suppose that Assumptions \ref{hyp: H convex} and \ref{hyp: V-uT Lipschitz 2nd D bounded DS} hold for some $c>0$ and $p>1$. Then,  for every $\varphi\in BV_0^+(\Omega_R)$
\begin{equation*}
\begin{split}
\Ww[\varphi]=\mathcal{J}[\varphi]&= \int_{\Omega_R}f_N\left( \frac{d(D_{t,x}\varphi)}{d \Ll^2}(t,x)\right)~\dx\dt-\int_{\Omega_R} (V+c(2p-1))\varphi_x -\int_{\Omega_R}u^\prime_T(x) \varphi_t    \\&+ \int_{\Omega_R}f_N\left( \frac{d(D_{t,x}\varphi)_s}{d ||(D_{t,x}\varphi)_s||}(t,x)\right) d||(D_{t,x}\varphi)_s||(t,x).
\end{split}
\end{equation*}
\end{theorem}

\begin{proof} 
	Taking into account \eqref{eq-lower-bound-for-L} and the definition of $f_N$, \eqref{def-f}, we have
	\begin{equation*}
		cp(|v_1|+v_2)\leq f_N(v_1,v_2), \quad (v_1,v_2)\in\Rr\times\Rr^+_0.
	\end{equation*}
Recalling that $f_N$ is convex and using the preceding estimate, the proof follows from Remark 5.37 in \cite{FoLe07} and Proposition \ref{pro-semi-lower2}.
\end{proof}

Next, relying on Theorem \ref{theorem-lower}, we state the relaxed variational problem.
We set
\begin{align*}
\mathcal{K}_0(\Omega_R)=& \left\{ \varphi \in BV_0^+(\Omega_R):~ \supp(\varphi_t(t,\cdot)),\supp(\varphi_x(t,\cdot))\subseteq (-R,R)\right\},
\\
\mathcal{K}(\Omega_R)= & \left\{ \varphi \in \mathcal{K}_0(\Omega_R):~ \varphi(0,x)=\int_{-R}^x m_0(y)\dy,~\varphi(t,-R)=0,\right.  
\\
& \quad \left.   ~ \int_{-R}^{R} \varphi_x(t,x)=1,~\int_{-R}^{R} \varphi(t,x)-M_0(x)~\dx=-\int_{0}^{t}Q(s) ~\ds, ~t\in[0,T]\right\}.
\end{align*}
Note that $\mathcal{A}(\Omega_R) \subset \mathcal{B}(\Omega_R) \subset \mathcal{K}(\Omega_R)$, so \eqref{eq:A(Omega_R) not empty} guarantees that $\mathcal{K}(\Omega_R)$ is a nonempty convex set. Our relaxed variational problem is
\begin{equation}\label{var-relaxed}
\min_{\varphi\in\mathcal{K}(\Omega_R)}	\Ii[\varphi],
\end{equation}
where 
\begin{align*}
\Ii[\varphi]= &\int_{\Omega_R}f_N\left( \frac{d(D_{t,x}\varphi)}{d \Ll^2}(t,x)\right)\dx\dt -\int_{\Omega_R}(V+c(2p-1))\varphi_x -\int_{\Omega_R}u^\prime_T(x) \varphi_t  
\\
& +\int_{\Omega_R}f_N\left( \frac{d(D_{t,x}\varphi)_s}{d ||(D_{t,x}\varphi)_s||}(t,x)\right) d||(D_{t,x}\varphi)_s||(t,x).
\end{align*}
The next theorem proves the existence of solutions to the preceding variational problem.
\begin{theorem}\label{thm:existence} Suppose that Assumptions \ref{hyp: grow-L}-\ref{hyp: V-uT convex DS} hold. 
Then, there exists $\varphi\in\mathcal{K}(\Omega_R)$ such that 
\begin{equation*}
\Ii[\varphi]=\min_{\psi\in\mathcal{K}(\Omega_R)}	\Ii[\psi].
\end{equation*}
\end{theorem}

\begin{proof} 
We recall that  $W^{1,1}(\Omega_R)$ is dense in $BV(\Omega_R)$ with respect to the intermediate convergence (see Theorem 10.1.2 in \cite{attouch2014variational}). Accordingly, we can take a minimizing sequence, $\{\varphi^n\}_{n=1}^\infty$,  such that  $\varphi^n\in W^{1,1}(\Omega_R)$. Therefore, 
\begin{equation*}
\min_{\psi\in\mathcal{K}(\Omega_R)}	\Ii[\psi]=\lim\limits_{n\to\infty}\Ii[\varphi^n] =\liminf\limits_{n\to\infty}I[\varphi^n],
\end{equation*}
where $I$, $\Ii$ are defined by \eqref{var-problem-c} and \eqref{var-relaxed}, respectively. 
Note that because 
\[
	\varphi^n_x\in \Pp(-R,R)\cap L^1(-R,R),
\]
there exists $\mu\in \Pp(-R,R)$, such that 
\begin{equation}\label{theo-phi-x}
||\varphi^n_x||_{L^1(\Omega_R)}\leq C,\quad\varphi^n_x\rightharpoonup\mu \text{ weakly in } \Mm(\Omega_R).
\end{equation}
Combining these estimates with the argument in Proposition \ref{pro-bound-L1}, we deduce that
\begin{equation}\label{theo-phi-t}
||\varphi^n_t||_{L^1(\Omega_R)}\leq C.
\end{equation}
Consequently, because $\Omega_R$ is bounded, Prohorov lemma (see Theorem 2.29 in \cite{Werner}) gives the existence of $\nu\in \Mm(\Omega_R)$ such that $\varphi^n_t\rightharpoonup\nu$ weakly in $\Mm(\Omega_R)$.
On the other hand, because $\int_{-R}^{R}\varphi(t,x)-M_0(x)~\dx=-\int_{0}^{t}Q(s) ~\ds$, we have that $\left| \int_{-R}^{R} \varphi^n\dx\right|\leq C$, where $C$ does not depend on $\varphi$. 
Hence, by Poincar{\'e} inequality (see Theorem 1 in Section 5.8.1 in \cite{E6}) from \eqref{theo-phi-x} and \eqref{theo-phi-t}, we get
 \begin{equation*}
||\varphi^n||_{L^1(\Omega_R)}\leq  ||\varphi^n_x||_{L^1(\Omega_R)}+||\varphi^n_t||_{L^1(\Omega_R)}+C\leq C.
 \end{equation*}
Therefore, $||\varphi^n||_{W^{1,1}(\Omega_R)}\leq C$. Consequently, Rellich-Kondrachov Theorem (see Theorem 1, Section 5.7 in \cite{E6}) implies that there exists $\varphi\in L^{\alpha}(\Omega_R)$ for $\alpha\in[1,2)$, such that $\varphi^n$ converges to $\varphi$ strongly in $ L^{\alpha}(\Omega_R)$. In particular, $\varphi^n$ converges to $\varphi$ strongly in $ L^{1}(\Omega_R)$.
This convergence combined with \eqref{theo-phi-x} and \eqref{theo-phi-t} implies that there exists $\varphi\in BV(\Omega_R)$, such that $\varphi^n\to\varphi$ in the sense of intermediate convergence in $BV(\Omega_R)$. 
Finally, relying on this  and recalling that $\varphi^n\in\mathcal{K}(\Omega_R)\cap W^{1,1}(\Omega_R)$ from  \cite[Theorem 10.2.2]{attouch2014variational}, we deduce that $\varphi\in\mathcal{K}(\Omega_R)$. 
Moreover, recalling the definition $f_N$ and using Propositions \ref{pro-semi-lower1} and \ref{pro-semi-lower2}, we get
\begin{equation*}
 \min_{\psi\in\mathcal{K}(\Omega_R)}	\Ii[\psi]=\lim\limits_{n\to\infty}\Ii[\varphi^n] =\liminf\limits_{n\to\infty}I[\varphi^n]\geq 	\Ii[\varphi]=	\min_{\psi\in\mathcal{K}(\Omega_R)}	\Ii[\psi]. \qedhere
 \end{equation*}
\end{proof}

\begin{remark}
Under Assumptions \ref{hyp: H convex}-\ref{hyp: V-uT convex DS}, Theorem 1 in \cite{gomes2018mean} implies that $\varphi$ given by \eqref{eq:Potential in terms of MFGs} belongs to $ \mathcal{K}(\Omega_R)$ and is a minimizer of \eqref{var-relaxed}. Thus, under Assumptions \ref{hyp: grow-L}-\ref{hyp: V-uT convex DS}, if we have uniqueness, $\varphi$ given by Theorem \ref{thm:existence}, and $\varphi$ given by \eqref{eq:Potential in terms of MFGs} coincide.

\end{remark}

\section{Price as Lagrange multiplier}\label{sec:Price as lagrangem}

In this section, we provide a representation formula for the price $\varpi$ using the minimizer $\varphi$ of \eqref{var-problem-c}. This formula shows that the Lagrange multiplier associated with the balance constraint \eqref{for-Poincare} characterizes the price. 
\begin{proposition}\label{pro-lag-mult} Suppose that Assumptions \ref{hyp: H convex}-\ref{hyp: V-uT convex DS} hold. 
Let $R$ satisfy \eqref{eq:R selection}. Let $(u,m,\varpi)$ solve \eqref{eq:MFG system} and let $\varphi\in \mathcal{A}(\Omega_{R})$ attain the minimum in \eqref{var-problem-c}. Furthermore, assume that $\varphi\in C^2(\Omega_R)$.  Then, $\varpi$ is given by a Lagrange multiplier $w:[0,T]\to \Rr$ associated with $\varphi$. 
\end{proposition}
\begin{proof}
The existence and uniqueness of the solution, $(u,m,\varpi)$, to Problem \ref{PMFG} follows from Theorem 1 in \cite{gomes2018mean}.  Because $L$ is convex by Remark \ref{rem-phi-in-u-m}, we have that $\varphi$ minimizes \eqref{functional-v}. Therefore, recalling Proposition \ref{pro-unique}, we deduce that $\varphi$ is the unique minimizer of \eqref{var-problem-c}.
Furthermore,  by Remark \ref{rem-phi-in-u-m} and Corollary \ref{corollary-m-comp} follows that there exists 
 $0<R_1<R$  such that 
\begin{equation}\label{eq:R1 selection}
	\supp(\varphi_x(t,\cdot))\subseteq[-R_1,R_1]\subseteq(-R,R), \quad t \in [0,T].
\end{equation}
Because $\varphi$ is the  minimizer of \eqref{var-problem-c}
	\begin{equation}\label{phi-prop}
	 \varphi_t(t,x)=0\quad a.e.\quad \text{in} \{\varphi_x(t,x)=0: (t,x)\in \Omega_R \}.
	\end{equation}
Let 
\begin{equation}\label{def-x_bar}
		\overline{x}(t):= \int_{\Rr} x \varphi_x(t,x) \dx=R-\int_{-R}^{R}\varphi(t,x)\dx, \quad t\in[0,T].
\end{equation}
Let $\psi\in W^{1,1}(\Omega_R)$ be such that $\psi(t,\cdot)$ is a cumulative distribution function on $(-R,R)$ for $t\in[0,T]$, and satisfies 
\begin{equation}\label{def-psi-prop}
	\psi(0,x)=\varphi(0,x),  \quad\supp(\psi_x(t,\cdot)) \subseteq\supp(\varphi_x(t,\cdot))\subseteq[-R_1,R_1]\subset(-R,R), \quad t\in[0,T],
\end{equation} 
and 
\begin{equation}\label{def-psi-prop2}
\psi_t(t,x)=0\quad a.e.\quad \text{in}\quad \{(t,x)\in \Omega_R:~ \psi_x(t,x)=0\}.
\end{equation}
Set
\begin{equation}\label{def-z_bar}
	\overline{z}(t):= \int_{\Rr} x \psi_x(t,x) \dx= R-\int_{-R}^{R}\psi(t,x)\dx, \quad t\in[0,T].
\end{equation}
Notice that $|\overline{x}-\overline{z}| \leq 4R$. Let $0<\varepsilon < \min\left\{\frac{R-R_1}{4R},1\right\}$. Thus,
\begin{equation}\label{eq-supp_psi_phi}
	[-R_1,R_1]\subset \left(-R-\varepsilon(\overline{x}(t)-\overline{z}(t)),R-\varepsilon(\overline{x}(t)-\overline{z}(t))\right)\cap(-R,R), \quad t \in [0,T].
\end{equation}
Let
\[
	\varphi^{\varepsilon}(t,x):=(1-\varepsilon)\varphi(t,x-\varepsilon(\overline{x}(t)-\overline{z}(t)))+\varepsilon\psi(t,x-\varepsilon(\overline{x}(t)-\overline{z}(t))).
\]
We claim that $\varphi^{\varepsilon}\in\mathcal{A}(\Omega_R)$. Indeed, by \eqref{def-psi-prop} and \eqref{eq-supp_psi_phi}, we have
\begin{equation*}
\begin{split}
	\varphi_x^{\varepsilon}\geq0, \quad \int_{-R}^{R}\varphi^{\varepsilon}_x\dx=1,& \quad\varphi^{\varepsilon}(0,x)=M_0(x)\quad x\in[-R,R],\\&\varphi^{\varepsilon}(t,-R)=0,\quad t\in[0,T].
\end{split}
\end{equation*}
It remains to prove that $\varphi^{\varepsilon}$ satisfies the balance condition; that is,
\begin{equation}\label{eq-bal-for-varphi_ep}
	\int_{-R}^{R}  \varphi^\varepsilon(t,x) \dx=-\int_{0}^{t}Q(s)ds+\int_{-R}^{R}  M_0(x) \dx, \quad t\in[0,T].
\end{equation}
Because $\varphi\in\mathcal{A}(\Omega_R)$, \eqref{def-x_bar} shows that to prove \eqref{eq-bal-for-varphi_ep} it is enough to verify that
\begin{equation*}
	\int_{\Rr} x \varphi^\varepsilon_x (t,x) \dx = \overline{x}(t), \quad t\in[0,T].
\end{equation*}
Computing the left-hand side of the previous identity, we have
\begin{align*}
	&\int_{\Rr} x \varphi^\varepsilon_x (t,x) \dx 
	\\
	& = (1-\varepsilon) \int_{\Rr} x \varphi_x (t,x-\varepsilon(\overline{x}(t)-\overline{z}(t))) \dx + \varepsilon \int_{\Rr} x \psi_x (t,x-\varepsilon(\overline{x}(t)-\overline{z}(t))) \dx
	\\
	& = (1-\varepsilon) \int_{\Rr} x \varphi_x (t,x) \dx + \varepsilon(\overline{x}(t)-\overline{z}(t))  + \varepsilon \int_{\Rr} x \psi_x (t,x) \dx + \varepsilon(\overline{x}(t)-\overline{z}(t)) 
	\\
	& = (1-\varepsilon) \left(\overline{x}(t) + \varepsilon(\overline{x}(t)-\overline{z}(t))\right) + \varepsilon \left( \overline{z}(t) + \varepsilon(\overline{x}(t)-\overline{z}(t))\right)
	\\
	& = \overline{x}(t).
\end{align*}
Therefore, $\varphi^{\varepsilon}\in\mathcal{A}(\Omega_R)$, and the map $\varepsilon \mapsto I[\varphi^\varepsilon ]$ has a minimum at $\varepsilon=0$; that is,
\[
	\lim_{\varepsilon\to 0^+} \frac{d}{d\varepsilon} I[\varphi^\varepsilon ] \geq 0.
\]
To compute the left-hand side of the previous inequality, we notice that
\begin{align*}
	 \varphi^\varepsilon_t(t,x) =& (1-\varepsilon)\bigg( \varphi_t(t,x-\varepsilon \left( \overline{x}(t)-\overline{z}(t)\right)) - \varphi_x(t,x-\varepsilon \left( \overline{x}(t)-\overline{z}(t)\right)) \varepsilon \left( \dot{\overline{x}}(t)-\dot{\overline{z}}(t)\right)\bigg)
	 \\
	&+\varepsilon\bigg( \psi_t(t,x-\varepsilon \left( \overline{x}(t)-\overline{z}(t)\right)) - \psi_x(t,x-\varepsilon \left( \overline{x}(t)-\overline{z}(t)\right)) \varepsilon \left( \dot{\overline{x}}(t)-\dot{\overline{z}}(t)\right)\bigg)  ,
	\\
	 \varphi^\varepsilon_x(t,x ) = & (1-\varepsilon)\varphi_x(t,x-\varepsilon \left( \overline{x}(t)-\overline{z}(t)\right)) +\varepsilon\psi_x(t,x-\varepsilon \left( \overline{x}(t)-\overline{z}(t)\right)). 
\end{align*}
Note that \eqref{phi-prop} and \eqref{def-psi-prop2} imply
\begin{equation}\label{def-varphiep-prop}
	\varphi^\varepsilon_t(t,x)=0\quad a.e.\quad \text{in}\quad \{\varphi^\varepsilon_x(t,x)=0: (t,x)\in \Omega_R \}.
\end{equation}
Furthermore,
\begin{align*}
	\lim_{\varepsilon\to 0^+} \dfrac{d}{d\varepsilon}\varphi^\varepsilon_t &= -\varphi_t(t,x) - \varphi_{tx} \left( \overline{x}(t)-\overline{z}(t)\right) - \varphi_x(t,x) \left( \dot{\overline{x}}(t)-\dot{\overline{z}}(t)\right) + \psi_t(t,x)
	\\
	&=\dfrac{d}{dt}\bigg(-\varphi(t,x) + \psi(t,x) - \varphi_x(t,x) \left( \overline{x}(t)-\overline{z}(t)\right) \bigg),
	\\
	\lim_{\varepsilon\to 0^+} \dfrac{d}{d\varepsilon}\varphi^\varepsilon_x  &= -\varphi_x(t,x) - \varphi_{xx} \left( \overline{x}(t)-\overline{z}(t)\right) + \psi_x(t,x)
	\\
	&=\dfrac{d}{dx}\bigg(-\varphi(t,x) + \psi(t,x) - \varphi_x(t,x) \left( \overline{x}(t)-\overline{z}(t)\right) \bigg).
\end{align*}
For ease of notation, we denote
	\begin{equation}\label{eq:Lzdef}
	L^*_z(t,x) =	\begin{cases}
	L_z(\varphi_t(t,x),\varphi_x(t,x)),\quad \varphi_x(t,x)>0\\
	0, \quad \text{otherwise},
		\end{cases}
	\end{equation}
and
	\begin{equation}\label{eq:Lydef}
L^*_y(t,x) =	\begin{cases}
		 L_y(\varphi_t(t,x),\varphi_x(t,x)),\quad \varphi_x(t,x)>0\\
		0, \quad \text{otherwise}.
	\end{cases}
\end{equation}
Taking into account \eqref{phi-prop} and \eqref{def-psi-prop2},  we obtain
\begin{align*}
& \lim_{\varepsilon\to 0^+} \frac{d}{d\varepsilon} I[\varphi^\varepsilon ] 
\\
&= \int_0^T \int_{-R}^{R} \left(L^*_z(t,x)-u'_T(x)\right)\dfrac{d}{dt}\bigg(\psi(t,x)-\varphi(t,x)-\varphi_x(t,x)\left( \overline{x}(t)-\overline{z}(t)\right) \bigg) \dx \dt
\\
& \quad + \int_0^T \int_{-R}^{R}\left(L^*_y(t,x)-V(x)\right)\dfrac{d}{dx}\bigg(\psi(t,x)-\varphi(t,x)-\varphi_x(t,x)\left( \overline{x}(t)-\overline{z}(t)\right) \bigg) \dx \dt.
\end{align*}
Integrating by parts in the right-hand side of the previous identity, using that $\overline{x}(0)=\overline{z}(0)$, and recalling \eqref{eq:R1 selection} and \eqref{def-psi-prop}, we obtain
\begin{align}\label{eq:Aux after integrating by parts}
& \lim_{\varepsilon\to 0^+} \frac{d}{d\varepsilon} I[\varphi^\varepsilon ]  \nonumber
\\
&= \int_{-R}^{R} \left(L^*_z(T,x)-u'_T(x)\right)\bigg(\psi(T,x)-\varphi(T,x)-\varphi_x(T,x)\left( \overline{x}(T)-\overline{z}(T)\right) \bigg) \dx \nonumber
\\
& \quad - \int_0^T \int_{-R}^{R}\bigg( \left(L^*_z(t,x)\right)_t+\left(L^*_y(t,x)-V(x)\right)_x \bigg) \nonumber
\\
& \quad \quad \quad \quad \quad \quad \bigg(\psi(t,x)-\varphi(t,x)-\varphi_x(t,x)\left( \overline{x}(t)-\overline{z}(t)\right) \bigg) \dx \dt.
\end{align}
Recalling \eqref{def-x_bar} and \eqref{def-z_bar}, we have
\begin{equation*}
\overline{x}(t)-\overline{z}(t)=\int_{-R}^{R} \left(\psi(t,x) - \varphi(t,x)\right)\dx, \quad t \in [0,T].
\end{equation*}
Using the previous identity, we write the first term on the right-hand side of \eqref{eq:Aux after integrating by parts} as follows
\begin{align}\label{eq:Aux after integrating by parts term1}
& \int_{-R}^{R} \left(L^*_z(T,x)-u'_T(x)\right)\bigg(\psi(T,x)-\varphi(T,x)-\varphi_x(T,x)\left( \overline{x}(T)-\overline{z}(T)\right) \bigg) \dx \nonumber
\\
& = \int_{-R}^{R} \bigg( L^*_z(T,x)-u'_T(x) - \int_{-R}^{R} (L^*_z(T,y)-u'_T(y))\varphi_x(T,y) \dy \bigg) \left( \psi(T,x)-\varphi(T,x)\right) \dx.
\end{align}
Similarly, the second term on the right-hand side of \eqref{eq:Aux after integrating by parts} becomes
\begin{align}\label{eq:Aux after integrating by parts term2}
&\int_0^T \int_{-R}^{R}\bigg( \left(L^*_z(t,x)\right)_t+\left(L^*_y(t,x)-V(x)\right)_x \bigg)\left(\psi(t,x)-\varphi(t,x)-\varphi_x(t,x)\left( \overline{x}(t)-\overline{z}(t)\right) \right) \dx \dt \nonumber
\\
&= \int_0^T \int_{-R}^{R} \bigg( \left(L^*_z(t,x)\right)_t +  \left(L^*_y(t,x)-V(x)\right)_x  \nonumber
\\
& \quad\quad\quad\quad\quad - \int_{-R}^{R} \bigg( \left(L^*_z(t,y)\right)_t +  \left(L^*_y(t,y)-V(y)\right)_x \bigg) \varphi_x(t,y) \dy \bigg) \left( \psi(t,x)-\varphi(t,x) \right) \dx \dt. 
\end{align}
Define the Lagrange multiplier by
\begin{align}\label{eq:Lagrange multiplier def}
& w_T = \int_{-R}^{R} (L^*_z(T,y)-u'_T(y))\varphi_x(T,y) \dy , \nonumber
\\
& w(t) = \int_{-R}^{R} \bigg( \left(L^*_z(t,y)\right)_t +  \left(L^*_y(t,y)-V(y)\right)_x \bigg) \varphi_x(t,y) \dy , \quad t\in[0,T].
\end{align}

Then, replacing \eqref{eq:Aux after integrating by parts term1} and \eqref{eq:Aux after integrating by parts term2} in \eqref{eq:Aux after integrating by parts}, we get
\begin{align}\label{eq:Aux first variation}
& \int_{-R}^{R} \bigg( L^*_z(T,x)-u'_T(x) - w_T\bigg) \left( \psi(T,x)-\varphi(T,x)\right) \dx \nonumber
\\
& + \int_0^T \int_{-R}^{R} \bigg( \left(L^*_z(t,x)\right)_t +  \left(L^*_y(t,x)-V(x)\right)_x - w(t) \bigg) \left( \psi(t,x)-\varphi(t,x) \right) \dx \dt \geq 0.
\end{align}
Notice that, in the previous inequality, the function $\phi=\psi-\varphi$ can be selected to be strictly positive or negative in any neighborhood of $(0,T)\times(-R,R)$. Therefore, we can infer the nullity of the functions in both integrals in \eqref{eq:Aux first variation} as follows. First, select $\psi$ satisfying $\psi(T,\cdot)=\varphi(T,\cdot)$. Then, \eqref{eq:Aux first variation} shows that 
\begin{align}\label{eq:Aux first variation conclusion1}
\int_0^T \int_{-R}^{R} \bigg( \left(L^*_z(t,x)\right)_t +  \left(L^*_y(t,x)-V(x)\right)_x - w(t) \bigg) \phi(t,x) \dx \dt \geq 0.
\end{align}
The regularity of 
\[
	(t,x)\mapsto \left(L^*_z(t,x)\right)_t +  \left(L^*_y(t,x)-V(x)\right)_x - w(t)
\]
allows the localization of the integral in \eqref{eq:Aux first variation conclusion1} using $\phi$, and we conclude that 
\begin{align}\label{eq:Aux first variation Integral conclusion1}
\left(L^*_z(t,x)\right)_t +  \left(L^*_y(t,x)-V(x)\right)_x - w(t) = 0 \quad \mbox{a.e. }(t,x) \in (0,T)\times (-R,R).
\end{align}
Then, \eqref{eq:Aux first variation} reduces to
\begin{align*}
& \int_{-R}^{R} \bigg( L^*_z(T,x)-u'_T(x) - w_T\bigg) \left( \psi(T,x)-\varphi(T,x)\right) \dx \geq 0,
\end{align*}
and we proceed as before by localizing the integral using $x \mapsto \phi(T,x)$ to conclude that 
\begin{align}\label{eq:Aux first variation Integral conclusion2}
L^*_z(T,x)-u'_T(x) - w_T = 0 \quad \mbox{a.e. }x \in (-R,R).
\end{align}
Recalling \eqref{eq-LE}, which characterizes $\varpi$, the identities \eqref{eq:Aux first variation Integral conclusion1} and \eqref{eq:Aux first variation Integral conclusion2} show that the price, $\varpi$, is given by the Lagrange multiplier \eqref{eq:Lagrange multiplier def} according to
\[
	\dot{\varpi}(t) = w(t), \quad t\in[0,T],\quad  \varpi(T)=w_T. \qedhere
\]
\end{proof}

\begin{proof}[Proof of Theorem \ref{pro-connection}] Because $\varphi$ solves Problem \ref{problem:Variational}, we obtain $\varpi$ according to Proposition \ref{pro-lag-mult}. Therefore, $(\varphi,\varpi)$ minimizes \eqref{eq: functional}, and by Proposition \ref{Pro: EL} satisfies an Euler-Lagrange equation equivalent to \eqref{eq: Euler-Lagrange wrt potential}-\eqref{boundary-in-phi}. Since the solution $(u,m,\varpi)$ of \eqref{eq:MFG system} defines a potential function according to \eqref{eq:Potential in terms of MFGs} which satisfies \eqref{eq: Euler-Lagrange wrt potential}-\eqref{boundary-in-phi}, the convexity of \eqref{eq: functional} implies that this potential is a minimizer of \eqref{eq: functional}. Thus, by Proposition \ref{pro-unique} , we conclude that the potential function defined by $(u,m,\varpi)$ coincides with the minimizer $\varphi$. Thus, we can recover $u$ and $m$ using \eqref{eq: potential relations}; that is,
\begin{equation}\label{eq:u as a function of phi}
	u(t,x) = u_T(x) - \int_t^T H\left(F'\left(\frac{\varphi_t(s,x)}{\varphi_x(s,x)}\right)\right) \ds - (T-t)V(x), \quad (t,x) \in [0,T]\times \Rr,
\end{equation}
where the right-hand side of the previous expression is well defined because $\varphi_x$ and $\varphi_t$ have the same compact support, and $m(t,x)=\varphi_x(t,x)$, $(t,x) \in [0,T]\times \Rr$. 
\end{proof}

\section{Numerical results}\label{sec:numerical results}

In this section, we provide the results of the potential approach applied to the price formation MFG system with quadratic cost and oscillating supply. We use the semi-explicit formulas introduced in \cite{gomes2018mean} to assess the error in our approximation. We use the standard solver for finite-dimensional convex problems provided by the software Mathematica to approximate the potential function in a discrete grid in time and space.

Let $\kappa\in \Rr$, $\eta\geq 0$, and $c>0$. For the quadratic cost configuration, we take
\[
	H(p) = \frac{1}{2c} p^2, \quad V(x) = -\frac{\eta}{2} \left( x - \kappa\right)^2, \quad \mbox{and} \quad u_T\left(x\right) \equiv 0.
\]
Thus, $F(v) = \frac{c}{2}v^2$. As shown in \cite{gomes2018mean} and \cite{gomes2021randomsupply}, a feature of the quadratic setting is the solvability of the Hamilton-Jacobi equation in \eqref{eq:MFG system} in the class of quadratic functions of $x$ with time-dependent coefficients 
\[
	u(t,x)=a_0(t) + a_1(t)x + a_2(t)x^2, \quad t\in[0,T],~ x \in \Rr.
\]
The coefficients $a_0$, $a_1$ and $a_2$ solve an ODE system that derives from the Hamilton-Jacobi equation by matching powers of the $x$ variable. Figure \ref{fig:value function u} shows the value function for $(x,t) \in [0,T]\times[-1,1]$. Moreover, the price has the following explicit formula
\begin{equation*}
	\varpi(t)=\eta \left(\kappa-\overline{m}_0\right)\left(T-t\right)- \eta \int_t^T \int_0^s Q(r) dr ds - c Q(t), \quad t\in[0,T],
\end{equation*}
where $\overline{m}_0 = \int_{\Rr} x m_0(x) dx$. The initial condition $m_0$ is centered at $x=0$ and with compact support $[-0.5,0.5]$ (see Figure \ref{fig:m0 plot}). The vector-field transporting $m_0$ is
\[
	b(t,x) = - \frac{1}{c}\left(\varpi(t) + a_1(t) + 2a_2(t)x\right), \quad t\in[0,T],~ x \in \Rr,
\]
which we use to compute $m$ using the method of characteristics (see Figure \ref{fig:m analsol supp}). Thus, recalling \eqref{eq: potential relations} and \eqref{eq:Potential in terms of MFGs}, we have explicit formulas for $\varphi$, $\varphi_x$, and $\varphi_t$. We use the previous expressions as a benchmark for the approximation obtained using \eqref{eq:Lagrange multiplier def}. 

For the discretization of the time variable, we set $T=1$ and $N_t = 20$ time steps uniformly spaced. Thus, $h_t = 0.05$ is the time step size. To discretize the space variable, the selection of $R$ in \eqref{eq:R selection}, where $R_0 = 0.5$, becomes
\[
	R > \max\{5.57742~,\quad 0.811579\}.
\]
However, to simplify the computational cost, we optimize the selection of $R$ by looking at the support of $m(t,)$ for $t\in[0,T]$, which we illustrate in Figure \ref{fig:m analsol supp}. Thus, we discretize the space variable in the space domain $[-1,1]$ using $N_x=40$ time steps equally spaced. Thus, $h_x =0.05$ is the step size. 

Because in several applications the supply function satisfies a mean reversion assumption, we assume that it follows the ordinary differential equation
\[
	\begin{cases}
	\dot{Q}(t) = \overline{Q}(t) - \alpha Q(t), & t \in [0,T],
	\\ 
	Q(0) = q_0,
	\end{cases}
\]
where $\overline{Q}:[0,T]\to \Rr$ represents the average supply over time, $\alpha \in \Rr$ measures the tendency to towards the average, and $q_0\in \Rr$ is the initial supply. For numerical purposes, we select
\[
	\overline{Q}(t) = 5 \sin(3 \pi t), \quad \alpha = 4, \quad q_0=-0.5.
\]
While the particular choice of $Q$ does not change the problem substantially, the preceding choice has oscillatory features, as we want to demonstrate how price changes and at the same time gives simple analytic expressions. As Figure \ref{fig:price and supply} shows, the price inherits the oscillating behavior from the supply. 

\begin{figure}[htp]
     \centering
     \begin{subfigure}[t]{0.34\textwidth}
		\vskip0cm         
         \centering        
         \includegraphics[width=\textwidth]{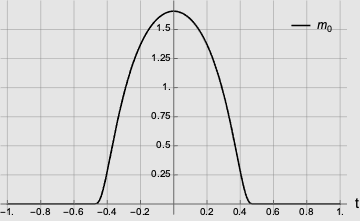}
         \caption{$m_0$}
         \label{fig:m0 plot}
     \end{subfigure}     
     \hfill
     \begin{subfigure}[t]{0.36\textwidth}
		\vskip0cm
         \centering        
         \includegraphics[width=\textwidth]{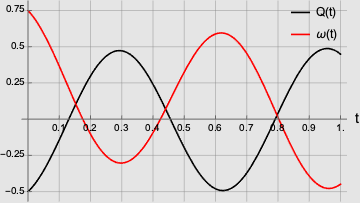}
         \caption{Price and supply}
         \label{fig:price and supply}
     \end{subfigure}
     \hfill
     \begin{subfigure}[t]{0.28\textwidth}
		\vskip0cm         
         \centering
         \includegraphics[width=\textwidth]{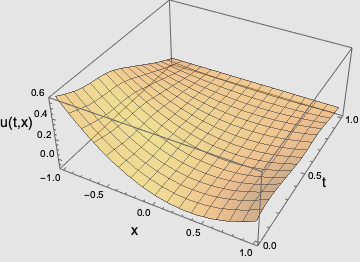}
        \caption{value function $u$}         
         \label{fig:value function u}
     \end{subfigure}
        \caption{Data $m_0$ and $Q$, and solutions $u$ and $\varpi$ for $\overline{Q}(t) = 5 \sin(3 \pi t)$.}
\end{figure}
Using the solution $(u,m,\varpi)$, we get $\varphi_t$, $\varphi_x$ from \eqref{eq: potential relations} (see Figure \ref{fig:m analsol supp}), and so \eqref{eq:Potential in terms of MFGs} gives $\varphi$, illustrated in Figure \ref{fig:Analsol phi}. The value of \eqref{var-problem-c} is $0.106525$, which we use as an additional benchmark to assess our numerical approximation.
\begin{figure}[htp]
     \centering
     \begin{subfigure}[t]{0.32\textwidth}
         \centering
         \includegraphics[width=\textwidth]{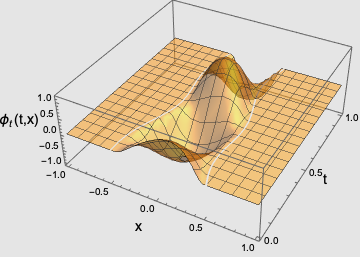}
         \caption{$\varphi_t$}
     \end{subfigure}
     \hfill
     \begin{subfigure}[t]{0.32\textwidth}
         \centering
         \includegraphics[width=\textwidth]{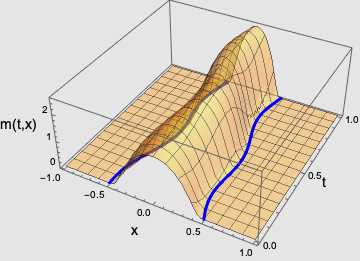}
        \caption{$\varphi_x=m$}         
         \label{fig:m analsol supp}
     \end{subfigure}   
     \hfill
     \begin{subfigure}[t]{0.32\textwidth}
         \centering
         \includegraphics[width=\textwidth]{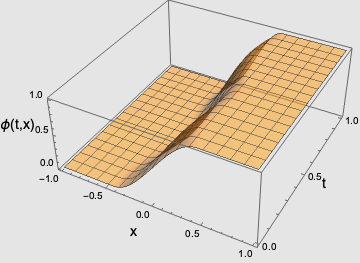}
        \caption{potential $\varphi$}         
     \end{subfigure}       
        \caption{Analytic solution $\varphi$ and its partial derivatives for $\overline{Q}(t) = 5 \sin(3 \pi t)$. The blue lines outline the support of $m$.}
        \label{fig:Analsol phi}
\end{figure}

We discretize \eqref{var-problem-c} over the time-space grid using finite differences to approximate $\varphi_t$ and $\varphi_x$; that is
\[
	\varphi_t(t_i,x_j)=\dfrac{\varphi(t_i+h_t,x_j)-\varphi(t_i,x_j)}{h_t}, \quad \varphi_x(t_i,x_j)=\dfrac{\varphi(t_i,x_j+h_x)-\varphi(t_i,x_j)}{h_x},
\]
for $i=1,\ldots,20,$ and $j=1,\ldots,40$. We obtain a finite-dimensional convex optimization problem with the following constraints
\begin{align*}
	& \varphi_x(t_i,x_j)\geq 0,\quad \sum_{j=1}^{N_x} \left(\varphi(t_i,x_j) - M_0(x_j)\right) h_x + \sum_{k=0}^{i} Q(t_k) h_t, 
	\\
	& \varphi(0,x_j) - M_0(x_j) = 0, \quad \varphi(t_i,-1)=0, \quad \varphi(t_i,1)=1, \quad i=1,\ldots,N_t, \; j=1,\ldots,N_x,
\end{align*}
which correspond to the discretization of the admissible set $\mathcal{A}(\Omega_R)$ (see \eqref{def:admissible-A}). The results are depicted in Figure \ref{fig:Applsol phi}. The approximated value of \eqref{var-problem-c} is $0.103765$, in good agreement with the theoretical value $0.106525$.

\begin{figure}[htp]
     \centering

     \begin{subfigure}[t]{0.34\textwidth}
         \centering
         \vskip0cm
         \includegraphics[width=\textwidth]{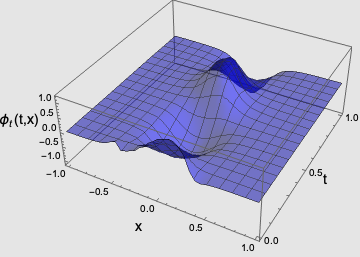}
        \caption{$\varphi_t$ approximation}         
     \end{subfigure}     
     \quad 
     \begin{subfigure}[t]{0.36\textwidth}
         \centering
         \vskip0cm         
         \includegraphics[width=\textwidth]{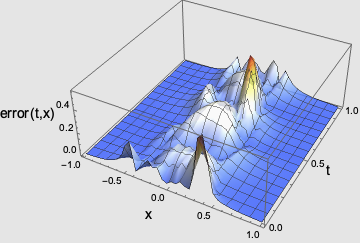}
         \caption{approximation error}         
     \end{subfigure}     
     
    \begin{subfigure}[t]{0.34\textwidth}
         \centering
         \vskip0cm         
         \includegraphics[width=\textwidth]{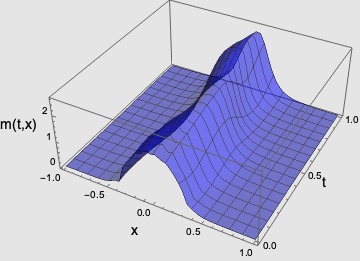}
        \caption{$\varphi_x=m$ approximation}         
     \end{subfigure}     
     \quad 
     \begin{subfigure}[t]{0.36\textwidth}
         \centering
         \vskip0cm         
         \includegraphics[width=\textwidth]{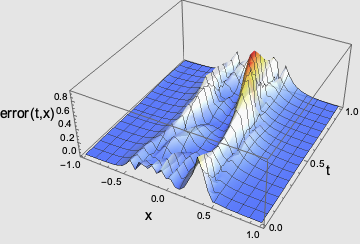}
         \caption{approximation error}         
     \end{subfigure}    
     
     \begin{subfigure}[t]{0.34\textwidth}
         \centering
         \vskip0cm         
         \includegraphics[width=\textwidth]{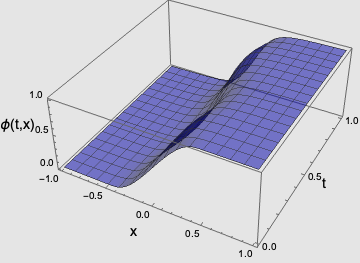}
        \caption{$\varphi$ approximation}         
     \end{subfigure}     
     \quad 
     \begin{subfigure}[t]{0.36\textwidth}
         \centering
         \vskip0cm         
         \includegraphics[width=\textwidth]{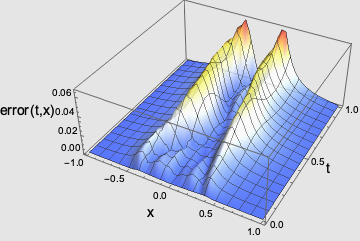}
         \caption{approximation error}         
     \end{subfigure}           
        \caption{Approximated solution $\varphi$ for $\overline{Q}(t) = 5 \sin(3 \pi t)$.}
        \label{fig:Applsol phi}
\end{figure}

Using \eqref{eq:Lagrange multiplier def}, we obtain the corresponding approximation of $\varpi$, illustrated in Figure \ref{fig:Applsol price}. Because of the implementation of finite differences, we can compute the price on the time horizon $[2 h_t,T]$. The plots show good agreement between the values of our numerical results with a small discrepancy that improves as the grid size increases (here, we show the results for the finest grid we used).

\begin{figure}[htp]
     \centering
     \begin{subfigure}[t]{0.38\textwidth}
         \centering
         \includegraphics[width=\textwidth]{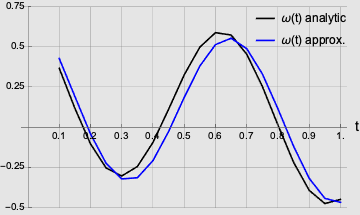}
        \caption{$\varpi$ approximation}         
     \end{subfigure}     
     \quad 
     \begin{subfigure}[t]{0.38\textwidth}
         \centering
         \includegraphics[width=\textwidth]{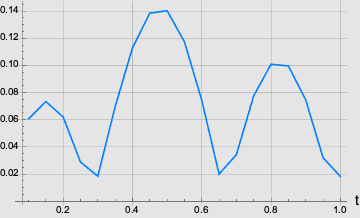}
         \caption{approximation error}         
     \end{subfigure}
        \caption{Approximated solution $\varpi$ for $\overline{Q}(t) = 5 \sin(3 \pi t)$.}
        \label{fig:Applsol price}
\end{figure}

As the last benchmark, we consider the value function $u$. To compute $u$, we round the approximation to avoid indeterminate expressions and we use \eqref{eq:u as a function of phi}. The result is depicted in Figure \ref{fig:Applsol u}. Again, we obtain a good agreement with the exact solution.

\begin{figure}[htp]
     \centering
     \begin{subfigure}[t]{0.34\textwidth}
         \centering
         \vskip0cm
         \includegraphics[width=\textwidth]{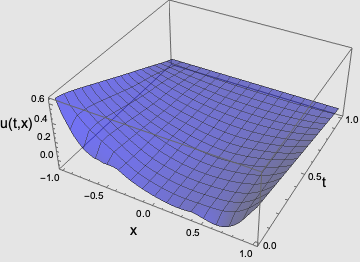}
        \caption{$u$ approximation}         
     \end{subfigure}     
     \quad 
     \begin{subfigure}[t]{0.36\textwidth}
         \centering
         \vskip0cm         
         \includegraphics[width=\textwidth]{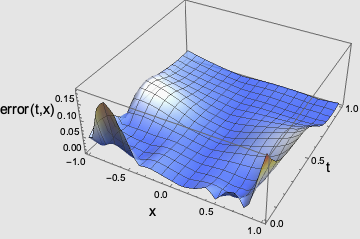}
         \caption{approximation error}         
     \end{subfigure}
        \caption{Approximated solution $u$ for $\overline{Q}(t) = 5 \sin(3 \pi t)$.}
        \label{fig:Applsol u}
\end{figure}

\section{Conclusions and further directions}

In this paper, we presented a variational approach based on Poincar\'{e} Lemma, reducing one variable in the MFG price formation model. We studied the variational approach independently of the MFG problem. We obtained existence for a relaxed formulation using bounded variation functions, and we proved uniqueness of the potential function. We showed that price existence follows a Lagrange multiplier rule associated with the balance constrained, an integral equation for the MFG model depending on a supply function. For the price problem, the variational formulation allows an efficient computation without solving the backward-forward coupled problem with integral constraints. The convexity of the variational approach allows the use of standard optimization tools to solve its discrete formulation. Our numerical method shows promising results and good agreement with the explicit solutions. We consider we can apply a similar approach to the price formation model with common noise, which corresponds to the case of a stochastic supply function. One challenge is the dependence of the variational problem formulation on the supply, requiring the discretization of time, state variables, and the common noise. We plan to investigate this case in future works.

\bibliographystyle{plain}

\bibliography{mfg.bib}





\end{document}